\documentclass[11pt,a4paper]{amsart}
\usepackage{mathrsfs}
\usepackage{amsmath}
\usepackage{amsthm}
\usepackage{amsfonts}
\usepackage{amssymb}
\usepackage{latexsym}
\usepackage{amscd,amssymb,amsopn,amsmath,amsthm,graphics,amsfonts,mathrsfs,accents,enumerate,verbatim,calc}
\usepackage[dvips]{graphicx}
\usepackage[colorlinks=true,linkcolor=red,citecolor=blue]{hyperref}
\usepackage[all]{xy}

\date{}
\allowdisplaybreaks[4] \footskip=15pt
\renewcommand{\uppercasenonmath}[1]{}

\numberwithin{equation}{section} \theoremstyle{plain}
\newtheorem{lem}{Lemma}[section]
\newtheorem{cor}[lem]{Corollary}
\newtheorem{prop}[lem]{Proposition}
\newtheorem{thm}[lem]{Theorem}

\newtheorem{Defn}[lem]{Definition}
\newtheorem{Ex}[lem]{Example}
\newtheorem{Quest}[lem]{Question}
\newtheorem{Property}[lem]{Property}
\newtheorem{Properties}[lem]{Properties}
\newtheorem{Subprops}{}[lem]
\newtheorem{Para}[lem]{}

\newtheorem{fact}[lem]{Fact}

\newtheorem{rem}[lem]{Remark}

\newenvironment{df}{\begin{Defn}\rm}{\end{Defn}}

\newenvironment{para}{\begin{Para}\rm}{\end{Para}}

\newtheorem*{ack*}{ACKNOWLEDGEMENTS}

\newcommand{\pf}{\noindent\begin {proof}}
\newcommand{\epf}{\end{proof}}

\newcommand{\C}{\mathbb{C}}
\newcommand{\Ab}{\mathcal{A}}

\newcommand{\AbT}{\widetilde{\mathcal{A}}}
\newcommand{\BBT}{\widetilde{\mathcal{B}}}

\newcommand{\GA}{\mathcal{G}(\mathcal{A})}

\newcommand{\s}{\stackrel}

\newcommand{\A}{\mathcal{A}}

\newcommand{\AT}{\widetilde{\mathcal{A}}}
\newcommand{\GAT}{\mathcal{G}(\mathcal{A})}

\newcommand{\BT}{\widetilde{\mathcal{B}}}
\newcommand{\AB}{\mathcal{A}\cap{\mathcal{B}}}

\newcommand{\Mod}{\mbox{\rm Mod}}
\newcommand{\B}{\mathcal{B}}

\newcommand{\Ext}{\mbox{\rm Ext}_R}

\newcommand{\oext}{\mbox{\rm Ext}}
\newcommand{\Hom}{\mbox{\rm Hom}_R}
\newcommand{\h}{\mbox{\rm H}}
\newcommand{\Yxt}{\overline{{\rm Ext}}}
\newcommand{\lHom}{\widetilde{{\rm Hom}}_{R}}
\newcommand{\uHom}{\overline{{\rm Hom}}_{R}}
\newcommand{\dHom}{\overline{{\rm Hom}}_{R}}
\newcommand{\wHom}{\widetilde{{\rm Hom}}_{R}}

\newcommand{\R}{\mathbf{R}}
\newcommand{\wtE}{\widetilde{{\rm Ext}}}
\newcommand{\wte}{\widetilde{{\rm Ext}}}
\newcommand{\AAB}{\mathbf{A}}
\newcommand{\ra}{\rightarrow}
\newcommand{\Ra}{\Rightarrow}

\newcommand{\im}{\mbox{\rm im}}
\newcommand{\coker}{\mbox{\rm coker}}

\newcommand{\Con}{\mbox{\rm Con}}

\begin{document}
\begin{center}
{\large  \bf  Tate-Vogel and relative cohomologies of complexes\\ with respect to cotorsion pairs}

\vspace{0.5cm}  Jiangsheng Hu$^{a}$, Huanhuan Li$^{b}$, Jiaqun Wei$^{c}$, Xiaoyan Yang$^{d}$ and Nanqing Ding$^{e}$  \\
$^{a}$Department of Mathematics, Jiangsu University of Technology,
Changzhou 213001, China\\

$^{b}$School of Mathematics and Statistics, Xidian University, Xi'an 710071, China\\

$^{c}$School of Mathematics Sciences, Nanjing Normal University, Nanjing 210046, China\\

$^{d}$Department of Mathematics, Northwest Normal University, Lanzhou 730070,
China\\
$^{e}$Department of Mathematics, Nanjing University, Nanjing 210093,
China
E-mails:jiangshenghu@jsut.edu.cn, lihuanhuan0416@163.com, weijiaqun@njnu.edu.cn, yangxy@nwnu.edu.cn and nqding@nju.edu.cn
\end{center}

\bigskip
\centerline { \bf  Abstract}
\leftskip10truemm \rightskip10truemm \noindent
We study Tate-Vogel and relative cohomologies of complexes by applying the model structure induced by a complete hereditary cotorsion pair ($\A$, $\B$) of modules. We show first that the class of complexes admitting a complete $\A$ resolution is exactly the class of complexes with finite Gorenstein $\A$ dimension. This lets us give general techniques for computing Tate-Vogel cohomoloies
of complexes with finite Gorenstein $\A$ dimension. As a consequence,
relative cohomology groups for complexes with finite Gorenstein $\A$ dimension are investigated. Finally, the relationships between  Gorenstein $\A$ dimensions and $\A$ dimensions for complexes are given. \\
\vbox to 0.3cm{}

\leftskip0truemm \rightskip0truemm
\bigskip
\section{ \bf Introduction}
Avramov and Foxby \cite{AFHd} explored projective and injective dimensions arising from constructions of $dg$-projective and $dg$-injective resolutions of complexes. Using the notions of $dg$-projective and $dg$-injective resolutions, one can define Gorenstein projective and Gorenstein injective dimensions for complexes (see \cite{IAG,vg}).
Cotorsion pairs were invented by Salce \cite{SL} in the category of
abelian groups, and rediscovered by Enochs and coauthors
\cite{BEE,EEEI,EJ0,EJ,EJT} in the 1900's.
Let $(\Ab,\B)$ be a complete hereditary cotorsion pair of modules.
A general construction of  $dg\AbT$ (resp., $dg\BBT$) resolutions of  any complex was provided in \cite[Corollary 3.8]{YL}. Furthermore, Yang and Ding \cite{YD} made a general study of $\mathcal{A}$ (resp., $\mathcal{B}$) dimensions of complexes. Recently, Yang and Chen \cite{YC} investigated Gorenstein $\A$ dimensions of complexes, which describe how Gorenstein dimensions of complexes should work for any complete hereditary cotorsion pair.

Tate cohomology was created in the 1950s, based on Tate's observation that the $\mathbb{Z}$G-module $\mathbb{Z}$ with the trivial action admits a complete projective resolution \cite{CE}. It was further extended by Avramov and Martsinkovsky  \cite{AM} to finitely generated modules of finite Gorenstein dimension over Noetherian rings and by Veliche \cite{vg} to
complexes of finite Gorenstein projective dimension.  On the other hand, as an attempt to generalize the theory to arbitrary groups, Tate-Vogel cohomology (also known as stable cohomology in \cite{AVV} to emphasize a relation to stabilization of module categories) was developed independently by Mislin \cite{Mislin}, Benson and
Carlson \cite{BF} and Vogel (first publish account in  \cite{Goichot}) and extended to complexes by Asadollahai and Salarian \cite{ASCT}. Recently, Hu and Ding \cite{HuDing} investigated Tate-Vogel cohomology of complexes by applying the model structure induced
by a complete hereditary cotorsion pair $(\Ab,\B)$ of modules. Detailed definition can be seen in Fact \ref{fact:2.5} below.

An interesting and deep result in \cite{ASCT} is that Tate-Vogel cohomology for complexes defined by  Asadollahai and Salarian is compatible with Tate cohomology for complexes defined by Veliche in \cite{vg}. More specifically, if $M$ is a complex of finite Gorenstein projective dimension and $N$ is a bounded below complex, then for each integer $i$ one has $$\wtE_R^i(M,N)\cong \h_{-i}(\Hom(T,N)),$$
where $\wtE_R^i(M,N)$ is a Tate-Vogel cohomology and $T\ra P\ra M$ is a complete resolution of $M$.

The main object of this paper is to study how Tate-Vogel cohomologies should work for any complex with finite Gorenstein $\A$ dimension whenever $(\Ab,\B)$ is a complete hereditary cotorsion pair of modules.

To state our main result more precisely, let us first introduce some definitions.

By $\Mod R$ (resp., $\mathbb{C}(R)$) we mean the category of all left $R$-modules (resp., all complexes of left $R$-modules). We always assume that $\AAB$ = ($\A$, $\B$) is a complete hereditary cotorsion pair in $\Mod R$.

Following \cite{CEG}, we denote by ${\C}(R)_{\mathcal{M}^{\B}_{\A}}$ the category $\mathbb{C}(R)$ endowed with the model structure induced by the complete hereditary cotorsion pair $(\Ab,\B)$ in $\Mod R$. Recall from \cite[Definition 3.1]{HuDing}, a \emph{cofibrant-fibrant} resolution of a complex $M$ is a diagram $\xymatrix@C=0.8cm{\mathcal{QR}M \ar[r]^{p_{\mathcal{R}M}} & \mathcal{R}M & M \ar[l]_{i_{M}}}$ of morphisms of complexes with $p_{\mathcal{R}M}$ a cofibrant replacement in ${\C}(R)_{\mathcal{M}^{\B}_{\A}}$ and $i_{M}$ a fibrant replacement in ${\C}(R)_{\mathcal{M}^{\B}_{\A}}$.

Let $M$ and $N$ be complexes of $R$-modules. Then we have the following exact sequence of complexes:
$$0\ra\uHom(\mathcal{QR}M,\mathcal{QR}N)\ra \Hom(\mathcal{QR}M,\mathcal{QR}N) \ra \wHom(\mathcal{QR}M,\mathcal{QR}N)\ra 0.$$
By \cite[Definition 3.7]{HuDing}, the $n$th \emph{Tate-Vogel cohomology group}, denoted by $\wtE_{\AAB}^n(M,N)$, is defined as
$$\wtE_{\AAB}^n(M,N)=\textrm{H}_{-n}(\wHom(\mathcal{QR}M,\mathcal{QR}N)).$$
Thus we have a long exact sequence $$\xymatrix@C=0.93cm{
  \cdots \ar[r]& \Yxt_{\AAB}^n(M,N) \ar[r] & \Ext^n(M,N) \ar[r]^{\widetilde{\varepsilon}_R^n(M,N)} & \wtE_{\AAB}^n(M,N)\ar[r] &\Yxt_{\AAB}^{n+1}(M,N) \ar[r]&\cdots,}$$
where $\Yxt_{\AAB}^n(M,N)=\h_{-n}(\uHom(\mathcal{QR}M,\mathcal{QR}N))$ for each integer $n$ (see Fact \ref{fact:2.5} and Remark \ref{rem:4.1} below).

Let $M$ be a complex. By \cite[Definition 4.1]{YC}, one has the notion of  Gorenstein $\A$ dimension of $M$. Detailed definition can be seen in Definition \ref{df:3.8'} below. It is proved that $M$ has finite Gorenstein $\A$ dimension if and only if  for each cofibrant-fibrant resolution $\mathcal{QR}M\ra \mathcal{R}M\leftarrow M$ of $M$, there exists a complete $\A$ resolution $T\s{\tau} \ra \mathcal{QR}M\ra \mathcal{R}M\leftarrow M$ of $M$ with each $\tau_{i}$ a split epimorphism such that $\tau_{i}$ $=$ {\rm id$_{T_{i}}$} for $i\gg0$ (see Definition \ref{df:3.8} and Theorem \ref{thm:3.9} below).

Now, our first main result can be stated as follows.

\begin{thm}\label{theorem:3.6} Let $R$ be a ring and $M$ a complex of finite Gorenstein $\A$ dimension. For each complex $N$ and each integer $n$, there exist isomorphisms
$$\wtE_{\AAB}^n(M,N)\cong \h_{-n}(\Hom(T,\mathcal{QR}N)) \ \textrm{and} \ \Yxt_{\AAB}^{n+1}(M,N)\cong\h_{-n}(\Hom(L_M,\mathcal{QR}N)),$$
 where $\mathcal{QR}N\ra \mathcal{R}N\leftarrow N$ is a cofibrant-fibrant resolution of $N$ and $T\s{\tau}\ra\mathcal{QR}M\ra \mathcal{R}M\leftarrow M$ is a split complete $\A$ resolution of $M$ with $L_M =\ker(\tau)$ such that each $\tau_{i}$ is a split epimorphism and $\tau_{i}$ $=$ {\rm id}$_{T_{i}}$ for all $i\gg0$. In the case when $N$ is a bounded below $dg$-$\B$ complex, we have  $$\wtE_{\AAB}^n(M,N)\cong \h_{-n}(\Hom(T,N)) \ \textrm{and} \ \Yxt_{\AAB}^{n+1}(M,N)\cong\h_{-n}(\Hom(L_M,N)).$$
\end{thm}

The relative cohomology theory was initiated by Butler and Horrocks \cite{BH} and Eilenberg and Moore \cite{EM} and has been revitalized recently by a number of authors (see, for example \cite{AM,EJ0,EJ,Hgd,sswG}), notably, Avramov and Martsinkovsky \cite{AM} and Enochs and Jenda \cite{EJ}. Based on the notions of proper $\mathcal{X}$-resolutions, one can define the relative cohomology functors $\textrm{Ext}_{\mathcal{X}}^{n}(-,-)$. Detailed definitions can be found in Definition \ref{df:2.0} below. As an application of Theorem \ref{theorem:3.6}, we have the following result. See \ref{4.10} for the proof.

\begin{thm}\label{thm:1.4'} Let  $M$ be an $R$-module with finite Gorenstein $\A$ dimension. Denote by $\GA$ the class of Gorenstein $\A$-modules. For any $R$-module $N$ in $\B$, we have the following isomorphisms:
\begin{enumerate}
\item {\rm $\oext^{1}_{\GA}(M,N)\cong\ker(\widetilde{\varepsilon}_R^1(M,N))$;
\item {\rm $\oext^{n}_{\GA}(M,N)\cong\Yxt_{\AAB}^{n}(M,N)$}} for any integer $n>1$.
\end{enumerate}
\end{thm}

We note that Theorem \ref{thm:1.4'} is motivated by \cite[Remark 6.7]{vg}, where the author pointed out
that there are some ``obstacles" to define relative cohomology groups for complexes. We refer to \cite{IAG,Liang,Liu2018} for a more discussion on this matter.
Theorem \ref{thm:1.4'}(2) shows that $\Yxt_{\AAB}^{n}(-,-)$ defined here extends the relative cohomology for modules with finite Gorenstein $\A$ dimension defined in Definition \ref{df:2.0} whenever $n$ is an integer with $n>1$.

%

By \cite[Definition 3.1]{YD}, the $\mathcal{A}$ dimension of a complex $M$, denoted by $\mathcal{A}$-dim$M$, is defined as $\mathcal{A}$-dim$M$ =
$\inf$\{sup\{$i$ $|$ $A_{-i}\neq0$\} $|$ $M\simeq A$ with $A\in{dg\widetilde{\mathcal{A}}}$\}. Note that each module in $\A$ is a Gorenstein $A$-module. It seems natural to investigate the relationships between  Gorenstein $\A$ dimensions and $\A$ dimensions for complexes. Motivated by this, we have the following theorem, which is a consequence of Theorems \ref{theorem:3.6} and \ref{thm:1.4'}. See \ref{proofof1.4} for the proof.

\begin{thm}\label{thm:4.4}Let {\rm ($\A$, $\B$)} be a complete hereditary cotorsion pair in $\Mod R$ and $M$ a complex.
\begin{enumerate}
\item There is an inequality
{\rm\begin{center}{$\GAT\textrm{-dim}M\leqslant \mathcal{A}\textrm{-dim}M,$}
\end{center}}
\noindent and the equality holds if {\rm$\mathcal{A}\textrm{-dim}M<\infty$}.

\item If {\rm $\GAT$-dim$M$ $<\infty$}, then the following are equivalent:
\begin{enumerate}
\item {\rm $\mathcal{A}\textrm{-dim}M$ = $\GAT$-dim$M$};
\item {\rm $\widetilde{\textrm{Ext}}_{\AAB}^i(M,X)=0$} for all $i\in{\mathbb{Z}}$ and all $dg$-$\B$ complexes $X$;
\item {\rm $\widetilde{\textrm{Ext}}_{\AAB}^i(M,N)=0$} for some $i\in{\mathbb{Z}}$ and all $R$-modules $N\in{\B}$;
\item For each cofibrant-fibrant resolution $\mathcal{QR}M\ra \mathcal{R}M\leftarrow M$ of $M$, there is a large enough $n$ such that
{\rm $\varepsilon^{i}(\textrm{C}_{n}(\mathcal{QR}M),N):\textrm{Ext}^{i}_{\GA}(\textrm{C}_{n}(\mathcal{QR}M),N)\ra \textrm{Ext}^{i}_{R}(\textrm{C}_{n}(\mathcal{QR}M),N)$} is an isomorphism for all $i\in{\mathbb{Z}}$ and any $R$-module $N$ in $\B$.
\end{enumerate}
\end{enumerate}
\end{thm}

The proof of Theorem \ref{thm:4.4} makes use of the fact that relative and Tate-Vogel cohomologies can be connected by a long
exact sequence (see Proposition \ref{thm:C'} below).

The contents of this paper are arranged as follows: Section \ref{pre} contains notations and definitions for use throughout this paper. Section \ref{G-dimensions} is devoted to proving Theorem \ref{theorem:3.6}, which is based on some characterizations of the finiteness of Gorenstein $\A$ dimensions of complexes (see Proposition \ref{thm:3.9}). In Section \ref{Relative-homology}, we study relative cohomology groups for complexes
with finite Gorenstein $\A$ dimension and prove Theorem \ref{thm:1.4'}. In Section \ref{comparsion}, we first show that relative and Tate-Vogel cohomologies can be connected by a long exact sequence (see Proposition \ref{thm:C'}), and then we give the proof of Theorem \ref{thm:4.4}.

\vspace{2mm}
\section{\bf Preliminaries}\label{pre}
 Throughout this paper, $R$ is an associative ring.
 All ``$R$-modules" and ``complexes" mean ``left $R$-modules" and ``chain complexes of left $R$-modules", respectively.
 We use the term ``subcategory" to mean a ``full and additive subcategory that is closed under isomorphisms". $\mathcal{E}$ is the class of exact complexes. $\mathcal{P}$, $\mathcal{F}$ and $\mathcal{C}$ denote the classes of projective, flat and cotorsion $R$-modules, respectively.

Next we recall basic definitions and properties needed in the sequel. For more details the reader can consult \cite{AFH}, \cite{CFH}, \cite{EJ}, \cite{Gillespie2016} or \cite{Happel}.

\vspace{2mm}
{\bf  Complexes.} Let $\mathcal{D}$ be an additive category. We denote by $\mathbb{C}(\mathcal{D})$ the category of complexes in $\mathcal{D}$; the objects are complexes and morphisms are chain maps. We write the complexes homologically, so an object $X$ of $\mathbb{C}(\mathcal{D})$ is of the form
$$X=\xymatrix@C=15pt@R=15pt{\cdots \ar[r]&  X_{n+1} \ar[r]^{\partial_{n+1}^X}&X_n \ar[r]^{\partial_n^X}& X_{n-1}\ar[r]&\cdots.&}$$
If $X_i=0$ for $i\neq0$ we identify $X$ with the object of $\mathcal{D}$ in degree 0, and an object $M$ in $\mathcal{D}$ is thought of as the stalk complex concentrated in degree zero. Let $X_{\leqslant n}$ be the complex with $i$th component equal to $X_i$ for $i\leqslant n$ and to 0 for $i>n$, and $X_{\geqslant n}$ be the complex with $i$th component equal to $X_i$ for $i\geqslant n$ and to 0 for $i<n$. The \emph{$i$th shift} of $X$ is the complex $X[i]$ with $n$th component $X_{n-i}$ and differential $\partial_n^{X[i]}=(-1)^{i}\partial_{n-i}^X$. The \emph{mapping cone} of a morphism $\varphi:X\ra Y$ is the complex $\Con(\varphi)$ defined by $\Con(\varphi)_n$ $=Y_n \oplus X_{n-1}$ and $\partial_n^{\textrm{Con}(\varphi)}= \left(
                                                                 \begin{array}{cc}
                                                                   \partial_n^Y & \varphi_{n-1} \\
                                                                   0 & -\partial_{n-1}^X \\
                                                                 \end{array}
                                                               \right).$

\vspace{1mm}
A \emph{homomorphism} $\varphi: X\ra Y$ of degree $n$ is a family of $(\varphi_{i})_{i\in{\mathbb{Z}}}$ of homomorphisms  $\varphi_i:X_i\ra Y_{i+n}$ in $\mathcal{D}$. In this case, we set $|\varphi|=n$.  All such homomorphisms form an abelian group, denoted by $\textrm{Hom}_{\mathcal{D}}(X,Y)_n$; it is clearly isomorphic to $\prod_{i\in{\mathbb{Z}}}\textrm{Hom}_{\mathcal{D}}(X_i,Y_{i+n})$. We let $\textrm{Hom}_{\mathcal{D}}(X,Y)$ be the complex of $\mathbb{Z}$-modules with $n$th component $\textrm{Hom}_{\mathcal{D}}(X,Y)_n$ and differential $$\partial(\varphi)=\partial^{Y}\varphi-(-1)^{|\varphi|}\varphi \partial^X.$$
If $\mathcal{D}= \Mod R$ is the category of (left) $R$-modules, we write $\textrm{Hom}_{R}(X,Y)$ for $\textrm{Hom}_{\Mod R}(X,Y)$ for all complexes $X$ and $Y$.

For any $i\in{\mathbb{Z}}$, the cycles in $\textrm{Hom}_{\mathcal{D}}(X,Y)_i$ are the \emph{chain maps} $X\ra Y$ of degree $i$. A chain map of degree 0 is a morphism. Two morphisms $\beta$ and $\beta'$ in $\textrm{Hom}_{\mathcal{D}}(X,Y)_0$ are called \emph{chain homotopic}, denoted by $\beta\sim \beta'$, if there exists a degree $1$ homomorphism $\nu$ such that $\partial(\nu)={\beta-\beta'}$. A \emph{chain homotopy equivalence} is a morphism $\varphi:X\ra Y$ for which there exists a morphism $\psi:Y\ra X$ such that $\varphi \psi\sim \textrm{id}_{Y}$ and $ \psi \varphi\sim \textrm{id}_{X}$.

The \emph{chain homotopy category} of $\mathcal{D}$ will be denoted by $\mathbb{K}(\mathcal{D})$. Its objects are the same as $\mathbb{C}(\mathcal{D})$ and morphisms are the chain homotopy classes of morphisms of complexes.

If $\mathcal{D}= \Mod R$ is the category of (left) $R$-modules, we write $\mathbb{C}(R)$ (resp., $\mathbb{K}(R)$) for $\mathbb{C}(\Mod R)$ (resp., $\mathbb{K}(\Mod R)$). It is known  that if $\mathcal{D}$ is additive (resp., abelian) then so is  $\mathbb{C}(\mathcal{D})$. In particular,  $\mathbb{C}(R)$ is an abelian category and $\mathbb{K}(R)$ is an additive category. We use subscripts $+,\ -, \ b$ to denote boundedness conditions. For example, $\mathbb{C}^{+}(R)$ is the full subcategory of $\mathbb{C}(R)$ of left bounded  (or bounded above) complexes.  To every complex
$$X=\xymatrix@C=15pt@R=15pt{\cdots \ar[r]&  X_{n+1} \ar[r]^{\partial_{n+1}^X}&X_n \ar[r]^{\partial_n^X}& X_{n-1}\ar[r]&\cdots&}$$
in $\mathbb{C}(R)$, the \emph{$n$th homology module} of $X$ is the module $\h_n(X)=\ker(\partial_n^X)/\im(\partial_{n+1}^X)$. We also set $\textrm{Z}_{n}(X)=\ker(\partial_{n}^X)$, $\textrm{B}_{n}(X)=\im(\partial_{n+1}^X)$ and $\textrm{C}_{n}(X)=\coker(\partial_{n+1}^X)$.

\vspace{1mm}
A \emph{quasi-isomorphism} $\varphi: X\ra Y$ with $X$ and $Y$ in $\mathbb{C}(R)$ is a morphism such that the induced map $\h_n(\varphi):\h_n(X)\ra \h_n(Y)$ is an isomorphism for all $n\in{\mathbb{Z}}$. The morphism $\varphi$ is a quasi-isomorphism if and only if $\Con(\varphi)$ is exact. Two complexes $X$ and $Y$ are \emph{equivalent} \cite[A.1.11, p.164]{chr}, and denoted by $X \simeq Y$, if they can be linked by a sequence of quasi-isomorphisms with arrows in the alternating directions.

Let $\mathcal{H}$ be a subcategory of $\Mod R$. Then a complex $L$ is \emph{$\Hom(\mathcal{H},-)$ exact} (resp., \emph{$\Hom(-,\mathcal{H})$ exact}) if the complex $\Hom(M,L)$ (resp., $\Hom(L,M)$) is exact for each $M\in{\mathcal{H}}$.

\vspace{2mm}
{\bf  Cotorsion pairs.} Let $\mathcal{D}$ be an abelian category and $\mathcal{X}$ a subcategory of $\mathcal{D}$. For an object $M\in\mathcal{D}$, write $M\in{^{\perp}\mathcal{X}}$ (resp., $M\in{^{\perp_1}\mathcal{X}}$) if $\textrm{Ext}_{\mathcal{D}}^{\geqslant1}(M,X)=0$ (resp., $\textrm{Ext}_{\mathcal{D}}^{1}(M,X)=0$) for each $X\in\mathcal{X}$. Dually, we can define $M\in{\mathcal{X}^{\perp}}$ and $M\in{\mathcal{X}^{\perp_1}}$.

\vspace{1mm}
Following Enochs \cite{EJ}, Hovey \cite{HC} and Salce \cite{SL}, a \emph{cotorsion pair} is a pair of classes ($\mathcal{A}$, $\mathcal{B}$) in $\mathcal{D}$ such that
$\mathcal{A}^{\perp_1}=\mathcal{B}$ and $^{\perp_1}
\mathcal{B}=\mathcal{A}$. A cotorsion pair ($\mathcal{A}$, $\mathcal{B}$) is said to be \emph{hereditary} \cite{GT} if
$\textrm{Ext}_{\mathcal{D}}^{i}(A,B)=0$ for all $i\geqslant 1$ and all $A\in{\mathcal{A}}$ and
$B\in{\mathcal{B}}$. If we restrict the abelian category $\mathcal{D}$ to the category of chain complexes of $R$-modules or the category of $R$-modules, then by \cite{HC}, the condition that ($\mathcal{A}$, $\mathcal{B}$) is hereditary is equivalent to that if whenever $0\ra L'\ra L\ra L''\ra 0$ is exact with $L$, $L''\in{\mathcal{A}}$, then $L'\in{\mathcal{A}}$, or equivalently, if whenever $0\ra B'\ra B\ra B''\ra 0$ is exact with $B'$, $B\in{\mathcal{B}}$, then $B''\in{\mathcal{B}}$.

\vspace{1mm}
Let $M$ be an object in $\mathcal{D}$. A morphism $\phi: M\rightarrow X$ with $X\in
\mathcal{X}$ is called an \emph{$\mathcal{X}$ preenvelope} of $M$ if for any morphism $f:M\rightarrow X'$ with
$X'\in\mathcal{X}$, there is a morphism $g:X\rightarrow X'$ such
that $g\phi=f$. A monomorphism $\phi: M\rightarrow B$ with $B\in{\mathcal{X}}$ is said to be a
\emph{special $\mathcal{X}$ preenvelope} of $M$ if ${\rm coker}(\phi)\in
{^{\perp_1}\mathcal{X}}$. Dually we have the definitions of an
$\mathcal{X}$ precover and a special $\mathcal{X}$ precover.

\vspace{1mm}
A cotorsion pair ($\mathcal{A}$,
$\mathcal{B}$) in $\mathcal{D}$ is called \emph{complete} if every object $M$ of
$\mathcal{D}$ has a special $\mathcal{B}$ preenvelope and a special
$\mathcal{A}$ precover. If we choose $\mathcal{D}=\Mod R$ for some
ring $R$, the most obvious example of a complete hereditary
cotorsion pair is ($\mathcal{P}$, $\Mod R$). Perhaps the most
useful complete hereditary cotorsion pair is the flat cotorsion pair
($\mathcal{F}$, $\mathcal{C}$). Here $\mathcal{C}$ is the collection
of all modules $C$ such that $C\in{\mathcal{F}^{\perp_1}}$. Such
modules are called cotorsion modules.

\vspace{1mm}
In \cite{HCc} Hovey laid out a correspondence between (nice enough) model structures on a bicomplete abelian category
$\mathcal{D}$ and cotorsion pairs on $\mathcal{D}$. Essentially, a model structure on $\mathcal{D}$ is two complete cotorsion pairs ($\mathcal{Q}$,
$\mathcal{R}\cap\mathcal{W}$) and ($\mathcal{Q}\cap\mathcal{W}$, $\mathcal{R}$), where $\mathcal{Q}$ is the class of cofibrant objects, $\mathcal{R}$ is the class of fibrant objects and $\mathcal{W}$ is the class of trivial objects. And a model structure on $\mathcal{D}$ is determined by the above cotorsion pairs in the following way: the (trivial) cofibrations are the monomorphisms with (trivially) cofibrant cokernel, the (trivial) fibrations are the epimorphisms with (trivially) fibrant kernel and the weak equivalences are the maps that can be factored as a trivial cofibration followed by a trivial fibration. We refer to \cite{HM} and \cite{HC} for a more detailed discussion on this matter.

\vspace{2mm}
{\bf  Derived Categories.} The \emph{derived category} of the category of chain complexes of $R$-modules, denoted by $\mathbb{D}(R)$, is the category of chain complexes of $R$-modules localized at the class of quasi-isomorphisms (see \cite{HR,Verdier}). The symbol ``$\simeq$" is used to designate isomorphisms in $\mathbb{D}(R)$. The homological position and size of a complex $X$ are captured by the numbers supremum and infimum defined by $$\textrm{sup}X=\textrm{sup}\{i\in{\mathbb{Z}} \ | \ \h_i(X)\neq0\},\ \textrm{inf}X=\textrm{inf}\{i\in{\mathbb{Z}} \ | \ \h_i(X)\neq0\}.$$
By convention sup$X=-\infty$ and $\inf X=\infty$ if $X\simeq 0$. The full subcategories $\mathbb{D}^{+}(R)$, $\mathbb{D}^{-}(R)$ and $\mathbb{D}^{b}(R)$ consist of complexes $X$ with $\h_{i}(X)=0$ for, respectively, $i\gg0$, $i\ll0$ and $|i|\gg0$.

Denote by $\R \Hom(-,-)$ the \emph{right derived functor} of the homomorphism functor of complexes; by \cite{AFHd} and \cite{SR} no boundedness conditions are needed on the arguments. That is, for $X$, $Y$ $\in{\mathbb{D}(R)}$, the complexes $\R \Hom(X,Y)$ are uniquely determined up to isomorphism in $\mathbb{D}(R)$, and they have the usual functorial  properties. We set $\Ext^i(X,Y)=\h_{-i}(\R \Hom(X,Y))$ for $i\in{\mathbb{Z}}$. For modules $X$ and $Y$ this agrees with the notation of classical homological algebra.

\begin{df}\label{df:2.2} {\rm (\cite{sswG})} For every complex $X$ in $\mathbb{C}(R)$ with $X_{-n}=0=\h_{n}(X)$ for all $n>0$, the natural morphism $X\ra M=\h_{0}(X)$ is a quasi-isomorphism. In this event, $X$ is an \emph{$\mathcal{X}$-resolution} of $M$ if each $X_n\in{\mathcal{X}}$, and the associated exact sequence
$$X^{+} =
\cdots\rightarrow X_{n}\rightarrow
\cdots \rightarrow X_{1}\rightarrow X_{0}\rightarrow M\rightarrow 0$$
is the \emph{augmented $\mathcal{X}$-resolution} of $M$ associated to $X$. An $\mathcal{X}$-resolution $X$ of $M$ is \emph{proper} if $X^+$ is $\Hom(\mathcal{X},-)$ exact.
The \emph{$\mathcal{X}$-projective dimension} of $M$ is the quantity
$$\mathcal{X}\textrm{-pd}(M)=\inf\{\textrm{sup}\{n \geq0 \ | \  X_{n}\neq0 \} \ | \ X \ \textrm{is} \ \textrm{an} \ \mathcal{X}\textrm{-resolution} \ \textrm{of} \ M\}.$$
\end{df}
%
%
%
%

\begin{df}\label{df:2.0}{\rm (\cite{sswG})} Let $\mathcal{X}$ be a class of $R$-modules and $M$ an $R$-module. If $f:X\ra M$ is a proper $\mathcal{X}$-resolution, then for each integer $n$ and each $R$-module $N$, the $n$th \emph{relative cohomology group} $\textrm{Ext}^{n}_{\mathcal{X}}(M,N)$ is
$$\textrm{Ext}^{n}_{\mathcal{X}}(M,N)=\h_{-n}(\Hom(X,N)).$$

We refer to \cite[Section 8.2]{EJ}, \cite[2.4]{Hgd} and \cite[Section 4]{sswG} for a detailed discussion on this matter.
\end{df}

\section {\bf Proof of Theorem \ref{theorem:3.6}}\label{G-dimensions}
The main object of this section is to prove Theorem \ref{theorem:3.6}. More precisely, we first recall the definition of Gorenstein $\A$ dimensions of complexes, and then give some characterizations of complexes with finite Gorenstein $\mathcal{A}$ dimension. Especially, we shall establish a result, Proposition \ref{thm:3.9}, which will play a crucial role in the proof of Theorem \ref{theorem:3.6}.

\begin{df}(\cite[Definition 3.1]{YC})\label{df:2.1'} Let $R$ be a ring. An $R$-module $M$ is a \emph{Gorenstein $\A$-module} if there is a $\Hom(-,{\AB})$ exact exact sequence $\cdots \ra X_1\ra X_0\ra X_{-1}\ra X_{-2}\ra \cdots $ with each $X_{i}\in {{\A}}$ such that $M\cong\ker(X_{0}\ra X_{-1})$.

In what follows, we write $\GA$ for the class of Gorenstein $\A$-modules.
\end{df}

\begin{rem}{\rm  By \cite[Lemma 3.2]{YC}, we get that an $R$-module $M$ is a Gorenstein $\A$-module if and only if  $M\in{^{\perp}(\AB)}$ and there is a $\Hom(-,\AB)$ exact exact sequence $0\rightarrow M\rightarrow X^{0}\rightarrow X^{1}\rightarrow \cdots$ with each $X^{i}\in {{\AB}}$.
}
\end{rem}

A class $\mathcal{X}$ of $R$-modules is called \emph{projectively resolving} \cite{hG} if
$\mathcal{P}\subseteq\mathcal{X}$ and for every exact
sequence $0\rightarrow X'\rightarrow X\rightarrow X''\rightarrow 0$ of $R$-modules
with $X''\in\mathcal{X}$ the conditions $X\in\mathcal{X}$ and
$X'\in\mathcal{X}$ are equivalent.

\begin{lem}\label{prop:2.3} The following are true for any ring $R$:
 \begin{enumerate}
 \item The class $\GA$ is projectively resolving. Furthermore, $\GA$ is closed under direct coproducts and direct summands.
 \item Let $0\rightarrow X\rightarrow Y\rightarrow Z\rightarrow0 $ be an exact sequence of $R$-modules. If $Y\in{\GA}$ and $X\in{\GA}$, then $Z\in{\GA}$ if and only if $Z\in{^{\bot_{1}}(\AB)}$.
\end{enumerate}
\end{lem}
\begin{proof}The result holds by \cite[Proposition 3.3]{chenwenjing} and \cite[Proposition 3.3]{YC}.
\end{proof}

\begin{df}\label{df:3.1}{\rm (}\cite[Definition 3.3]{GT}{\rm )} {\rm Let $R$ be ring and $X$ a complex.
\begin{enumerate}
\item $X$ is called an $\A$ complex if it is exact and $\textrm{Z}_{n}X$ $\in$ $\A$ for all $n$.
\item $X$ is called a $\B$ complex if it is exact and $\textrm{Z}_{n}X$ $\in$ $\B$ for all $n$.
\item $X$ is called a $dg$-$\A$ complex if $X_{n}\in{\A}$ for each $n$, and $\Hom(X,B)$ is exact whenever $B$ is a $\B$ complex.
\item $X$ is called a $dg$-$\B$ complex if $X_{n}\in{\B}$ for each $n$, and $\Hom(A,X)$ is exact whenever $A$ is an $\A$ complex.
\end{enumerate} }
\end{df}

In the following, we denote the class of $\A$ (resp., $\B$) complexes by $\widetilde{\A}$ (resp., $\widetilde{\B}$) and the class of $dg$-$\A$ (resp., $dg$-$\B$) complexes by $dg\AT$ (resp.,
$dg\BT$). By \cite[Theorem 3.12]{GT}, we have
$dg$$\AT$ $\cap$ $\mathcal{E}$ $=$ $\AT$ and $dg$$\BT$ $\cap$
$\mathcal{E}$ $=$ $\BT$ whenever ($\A$, $\B$) is a complete hereditary cotorsion pair in $\Mod R$. In particular, if ($\A$, $\B$) $=$ ($\mathcal{P}$, $\Mod R$)  then $dg$-$\mathcal{P}$ complexes are exactly \emph{$dg$-projective} complexes (see \cite{AFHd,CFH}). We refer to \cite{AFHd}, \cite{CFH} and
\cite{GT} for a more detailed discussion on this matter.

\begin{lem}\label{lem:3.3}{\rm (}\cite[Corollary 2.8]{YL}{\rm )} If {\rm ($\mathcal{A}$, $\mathcal{B}$)} is a complete hereditary cotorsion pair in $\Mod R$, then the induced cotorsion pairs {\rm ($\widetilde{\mathcal{A}}$, $dg\widetilde{\mathcal{B}}$)} and {\rm ($dg$$\widetilde{\mathcal{A}}$, $\widetilde{\mathcal{B}}$)} in ${\C}(R)$ are both complete and hereditary. Furthermore, $dg$$\widetilde{\mathcal{A}}$ $\cap$ $\mathcal{E}$ $=$ $\widetilde{\mathcal{A}}$ and $dg\widetilde{\mathcal{B}}$ $\cap$
$\mathcal{E}$ $=$ $\widetilde{\mathcal{B}}$, where $\mathcal{E}$ is the class of exact complexes. So there exists a model structure on ${\C}(R)$, denoted by  {\rm${\C}(R)_{\mathcal{M}^{\B}_{\A}}$}, satisfying that:
\begin{enumerate}
\item the weak equivalences are the quasi-isomorphisms;
\item the cofibrations {\rm(}resp., trivial cofibrations{\rm)} are the monomorphisms whose
cokernels are in $dg\widetilde{\mathcal{A}}$ {\rm(}resp., $\widetilde{\mathcal{A}}${\rm)};
\item the fibrations {\rm(}resp., trivial fibrations{\rm)} are the epimorphisms whose kernels
are in $dg$$\widetilde{\mathcal{B}}$  {\rm(}resp., $\widetilde{\mathcal{B}}${\rm)}.
\end{enumerate}

In particular, $dg$$\widetilde{\mathcal{A}}$ is the class of cofibrant objects and $dg$$\widetilde{\mathcal{B}}$ is the class of fibrant objects.
\end{lem}
Some nice introductions to the basic ideal of a model category can be found in \cite{DS95,HM}.
\begin{fact}\label{fact:2.3}{\rm Let $M$ be a complex.  Then $M$ has a cofibrant replacement $p_{M}:\mathcal{Q}M\ra M$ in
${\C}(R)_{\mathcal{M}^{\B}_{\A}}$ and a fibrant replacement $i_{M}:M\ra \mathcal{R}M$ in ${\C}(R)_{\mathcal{M}^{\B}_{\A}}$, where $\mathcal{Q}M$ is cofibrant and $p_{M}$ is a trivial fibration, and
$\mathcal{R}M$ is fibrant and $i_{M}$ is a trivial cofibration. We can insist these exist functorially if one wishes. We
refer to \cite[Sections 4 and 5]{DS95} and \cite[Section 4]{Gillespie} for a detailed discussion on this
matter.}
\end{fact}

\begin{df}\label{df:2.3}{\rm (}\cite[Definition 3.1]{HuDing}{\rm )}  Let $R$ be a ring and $M$ a complex.
A \emph{cofibrant-fibrant} resolution of $M$ is a diagram $\xymatrix@C=0.8cm{\mathcal{QR}M \ar[r]^{p_{\mathcal{R}M}} & \mathcal{R}M & M \ar[l]_{i_{M}}}$ of morphisms of complexes with $p_{\mathcal{R}M}$ a cofibrant replacement in ${\C}(R)_{\mathcal{M}^{\B}_{\A}}$ and $i_{M}$ a
fibrant replacement in ${\C}(R)_{\mathcal{M}^{\B}_{\A}}$.
\end{df}
\begin{rem}\label{rem:3.2}{\rm  We note that $\mathcal{QR}M$ in the above definition is in  $dg\AT \cap dg\BT$. As far as the notions of special $dg\AT$ precovers and special $dg\BT$ preenvelopes are concerned, a cofibrant replacement $p_{M}$ is exactly a special  $dg\AT$ precover of $M$ and a fibrant replacement $i_{M}$ is exactly a special  $dg\BT$ preenvelope of $M$.}
\end{rem}

\begin{df}(\cite[Definition 4.1]{YC}) \label{df:3.8'}{\rm Let $R$ be a ring and $M$ a complex.
\begin{enumerate}
\item A $dg\AT$-resolution of a complex $M$ is a quasi-isomorphism $X\ra M$ with $X\in{dg\AT}$;

 \item A $dg\AT$-resolution $X\ra M$ is proper if $\textrm{Hom}_{R}(X',X)\ra \textrm{Hom}_{R}(X',M)\ra 0$ is exact for any $X'\in{dg\AT}$.

 \item The \emph{Gorenstein $\A$ dimension} of $M$, denoted by $\mathcal{GA}$-dim$M$, is less than or equal to $n$, if there exists a proper $dg\AT$-resolution $X\ra M$ with $\sup(X)\leqslant n$ and $\textrm{C}_{n}(X)\in{\GA}$. If $\mathcal{GA}$-dim$M$ $\leqslant n$ for all integers $n$ then $\mathcal{GA}$-dim$M$ $=-\infty$; if $\mathcal{GA}$-dim$M$ $\leqslant n$ does not hold for any $n$ then $\mathcal{GA}$-dim$M$ $=\infty$.
\end{enumerate}}
\end{df}

\begin{df}\label{df:3.8}{\rm Let $R$ be a ring.
\begin{enumerate}
\item An exact complex $X$ is called \emph{totally $\A$-acyclic} if each entry of $X$ belongs to $\AB$ and $\textrm{C}_{i}(X)$ is a Gorenstein $\A$-module for every $i\in \mathbb{Z}$.
\item A \emph{complete $\A$ resolution} of a complex $M$ is a diagram  $T\s{\tau} \ra \mathcal{QR}M\ra \mathcal{R}M\leftarrow M$ of morphisms of complexes with $\mathcal{QR}M\ra \mathcal{R}M\leftarrow M$ a cofibrant-fibrant resolution of $M$ such that $T$ is a totally $\A$-acyclic complex and $\tau_{i}$ is bijective for all $i\gg0$. Moreover, a complete $\A$ resolution resolution is \emph{split} if $\tau_{i}$ is a split epimorphism for all $i\in{\mathbb{Z}}$.
\end{enumerate}}
\end{df}

\begin{rem} {\rm Let $M$ be a complex.
\begin{enumerate}
\item If we replace the cotorsion pair ($\A$, $\B$) with the
cotorsion pair ($\mathcal{P}$, $\Mod R$) in the definition above,
then one easily checks that the complete
$\mathcal{P}$ resolution of $M$ defined here is the complete resolution
of $M$ given by Veliche in \cite[Definition 2.2.1]{vg}.

\item If we replace the cotorsion pair ($\A$, $\B$) with the flat
cotorsion pair ($\mathcal{F}$, $\mathcal{C}$) in the definition above, then one can show that the complete $\mathcal{F}$ resolution of $M$ here
is the complete flat resolution
of $M$ defined by Hu and Ding in \cite[Definition 5.3]{HuDing}.
\end{enumerate}}
\end{rem}

\begin{lem}\label{lem:3.3''} Let $T\s{\tau}\ra \mathcal{QR}M\ra \mathcal{R}M\leftarrow M$ be a complete $\A$ resolution of $M$. If $n$ is an integer such that $\tau_{i}$ is bijective for all $i\geqslant n$, then there exist a complete $\A$ resolution $T'\s{\tau'}\ra \mathcal{QR}M\ra \mathcal{R}M\leftarrow M$ with each $\tau'_{i}$ a split epimorphism such that $\tau'_{i}$ is bijective for all $i\geqslant n$ and a homotopy equivalence $\alpha:T\ra T'$ such that $\tau=\tau' \alpha$ and {\rm$\alpha_{i}=\textrm{id}_{T_{i}}$} for all $i\geqslant n$.
\end{lem}
\begin{proof} The proof is similar to that of \cite[Construction 3.7]{AM}.
\end{proof}

\begin{lem}\label{lem:3.4}{\rm(}\cite[Lemma 3.5]{HuDing}{\rm)} If $F\in$ $dg\AT$ and $X\in$ $dg\BT$, then $\R\Hom(F,X)$ can be represented by $\Hom(F,X)$.
\end{lem}

\begin{lem}\label{prop:3.5} If $A\simeq P$ in $\mathbb{D}^{+}(R)$ with $P$ a $dg$-projective complex and $A$ a $dg$-$\A$ complex, then
$$\Ext^{1}({\rm C}_{n}(A),B)\cong\Ext^{1}({\rm C}_{n}(P),B)$$
\noindent for any $B\in{\B}$ and any integer $n\geqslant$ $\sup P$.
\end{lem}
\begin{proof} By assumption, we get that sup$P$ $=$ sup$A$ $\leqslant n<\infty$. For each $B\in{\B}$, we have

\vspace{2mm}
\hspace{36mm}$\Ext^{1}({\rm C}_{n}(A),B) = \h_{-1}(\R\Hom({\rm C}_{n}(A),B))$

\vspace{1mm}
\hspace{64mm}$\cong\h_{-1}(\R\Hom(\Sigma^{-n}(A_{\geqslant n}),B))$

\vspace{1mm}
\hspace{64mm}$=\h_{-n-1}(\R\Hom(A_{\geqslant n},B))$

\vspace{1mm}
\hspace{64mm}$\cong\h_{-n-1}(\Hom(A_{\geqslant n},B))$

\vspace{1mm}
\hspace{64mm}$=\h_{-n-1}(\Hom(A,B))$,

\vspace{1mm}
\hspace{64mm}$\cong\h_{-n-1}(\R\Hom(A,B))$

\vspace{1mm}
\hspace{64mm}$\cong\h_{-n-1}(\Hom(P,B))$

\vspace{1mm}
\hspace{64mm}$=\h_{-1}(\Hom(\Sigma^{-n}(P_{\geqslant n}),B))$

\vspace{1mm}
\hspace{64mm}$\cong\h_{-1}(\R\Hom({\rm C}_{n}(P),B))$

\vspace{1mm}
\hspace{64mm}$=\Ext^{1}({\rm C}_{n}(P),B)$,

\vspace{2mm}
\noindent where the second and third isomorphisms follow from \cite[Lemma 3.4]{GT} and Lemma \ref{lem:3.4}. This completes the proof.
\end{proof}

The next result parallels \cite[Proposition 5.4]{HuDing}, which is a key result to prove Theorem \ref{theorem:3.6}.

\begin{prop}\label{thm:3.9} Let $M$ be a complex. Then the following are equivalent for each integer $n$:
\begin{enumerate}
\item {\rm$\GAT$-dim$M$ $\leqslant n$};

\item {\rm sup$M$} $\leqslant n$ and for every $dg\AT$-resolution $X\ra M$  the object  {\rm$\textrm{C}_{n}$($X$) $\in{\GA}$};

\item {\rm sup$M$} $\leqslant n$ and {\rm$\textrm{C}_{n}$($P$)} $\in{\GA}$ for any $P\simeq M$ with $P$ $dg$-projective;

\item {\rm sup$M$} $\leqslant n$ and there exists a $dg$-projective complex $P$ with $P\simeq M$ such that {\rm$\textrm{C}_{n}$($P$) $\in{\GA}$};

\item {\rm sup$M$} $\leqslant n$ and {\rm$\textrm{C}_{n}$($X$) $\in{\GA}$} for any $X\simeq M$ with $X$ a $dg$-$\A$ complex;

\item  For each cofibrant-fibrant resolution $\mathcal{QR}M\ra \mathcal{R}M\leftarrow M$ of $M$, there exists a complete $\A$ resolution $T\s{\tau} \ra \mathcal{QR}M\ra \mathcal{R}M\leftarrow M$ of $M$ with each $\tau_{i}$ a split epimorphism such that $\tau_{i}$ $=$ {\rm id$_{T_{i}}$} for $i\geqslant n$.
\end{enumerate}
\noindent Furthermore, if {\rm $\GAT$-$\textrm{dim}M < \infty$}, then
\vspace{2mm}
{\rm\begin{center}
{$\GAT$-$\textrm{dim}M=\textrm{sup}\{-\inf\R\Hom(M,X) \ | \
X\in{\mathcal{A}\cap\mathcal{B}} \}.$ }
\end{center}}
\end{prop}

\begin{proof} 
$(1)\Leftrightarrow(2)$ holds by \cite[Proposition 4.2]{YC}.

$(2)\Rightarrow(3)$ follows from \cite[1.4.P, p.132]{AFHd}.

$(3)\Rightarrow(4)$ is trivial.

$(4)\Rightarrow(5)$. By hypothesis, there exists a $dg$-projective complex $P$ with $P\simeq M$ such that {\rm$\textrm{C}_{n}$($P$) $\in{\GA}$}. Let $A$ be a $dg$-$\A$ complex with $A\simeq M$. Thus $A\simeq P$, and hence there is a quasi-isomorphism $\varphi:P\ra A$ by \cite[1.4.P, p.132]{AFHd}.

If $\varphi$ is surjective, then there is an exact sequence $0\ra K\ra P\s{\varphi}\ra A\ra 0$ of complexes with $K$ an exact complex. Since both $P$ and $A$ are $dg$-$\A$ complexes, $K$ is a $dg$-$\A$ complex by Lemma \ref{lem:3.3}.  Let $N$ be any $R$-module in $\AB$. Then $\Ext^{1}(\textrm{C}_{n}(P),N)=0$ by hypothesis. By Lemma \ref{prop:3.5}, we get that $\Ext^{1}(\textrm{C}_{n}(A),N)=0$.
Note that both $P$ and $A$ are $dg$-$\A$ complexes and $K$ is an exact complex. Then $K\in{\AT}$ and $\textrm{C}_{n}(K)\in{\A}$. Thus we have an exact sequence $0\ra \textrm{C}_{n}(K) \ra \textrm{C}_{n}(P)\ra \textrm{C}_{n}(A)\ra 0 $ with $\textrm{C}_{n}(K)\in{\A}$  and $\textrm{C}_{n}(P)\in{\GA}$, and so $\textrm{C}_{n}(A)\in{\GA}$ by Lemma \ref{prop:2.3}(2).

Suppose that $\varphi$ is not surjective. By Lemma \ref{lem:3.3}, there is a special $\widetilde{\A}$ precover $G\ra A$ of $A$. Thus ${P\oplus G}\ra A$ is a surjective quasi-isomorphism with $P\oplus G\in{dg\AT}$ such that $\textrm{C}_{n}(P\oplus G)\cong {\textrm{C}_{n}(P)\oplus \textrm{C}_{n}(G)}$. Note that $\textrm{C}_{n}(P\oplus G)\in{\GA}$ by the proof above, and so $\textrm{C}_{n}(P)\in{\GA}$ by Lemma \ref{prop:2.3}(1).

$(5)\Rightarrow(2)$ is trivial.

$(5)\Ra(6)$.  By Lemma \ref{lem:3.3''}, it suffices to show that there exists a complete $\A$ resolution $T\s{\tau}\ra \mathcal{QR}M\ra \mathcal{R}M\leftarrow M$ of $M$ with $\tau_{i}$ $=$ id$_{T_{i}}$ for $i\geqslant n$.
By hypothesis, there exists a cofibrant-fibrant resolution $\mathcal{QR}M\ra \mathcal{R}M\leftarrow M$ of $M$
such that $\textrm{C}_{n}$($\mathcal{QR}M$) $\in{\GA}$.
Then there exists a $\Hom(-,\AB)$ exact exact sequence $ 0\rightarrow \textrm{C}_{n}(\mathcal{QR}M)\rightarrow (\mathcal{QR}M)_{0}\rightarrow (\mathcal{QR}M)_{-1}\rightarrow \cdots$ with each $(\mathcal{QR}M)_{i}\in {\AB}$.
Let $\widehat{X}$ $=$ $\Sigma^{n-1} X$ where $X$ is the complex $0\rightarrow (\mathcal{QR}M)_{0}\rightarrow (\mathcal{QR}M)_{-1}\rightarrow \cdots$.
Note that $\mathcal{QR}M\in{dg\AT \cap dg\BT}$.
Thus there exists a morphism $\gamma:\widehat{X}\rightarrow (\mathcal{QR}M)_{\leqslant {n-1}}$ such that the following diagram
$$\xymatrix{0\ar[r]&\textrm{C}_{n}(\mathcal{QR}M)\ar@{=}[d]\ar[r]&\widehat{X}_{n-1}\ar[r] \ar[d]^{\gamma_{n-1}}&\widehat{X}_{n-2}\ar[r] \ar[d]^{\gamma_{n-2}}&\cdots\\
0\ar[r]&\textrm{C}_{n}(\mathcal{QR}M)\ar[r]&(\mathcal{QR}M)_{n-1}\ar[r] &(\mathcal{QR}M)_{n-2}\ar[r]
&\cdots}$$
commutes. Let $T$ be the complex obtained by splicing
$(\mathcal{QR}M)_{\geqslant n}$ and $\widehat{X}$ along $\textrm{C}_{n}(\mathcal{QR}M)$. One
easily checks that $T$ is totally $\A$-acyclic. Set
$$\tau_i = \bigg\{\begin{array}{cc}
                        \gamma_i & \textrm{for} \ i<n, \\
                        \textrm{id}_{(\mathcal{QR}M)_{i}} & \textrm{for} \ i\geqslant n.
                      \end{array}$$
Thus $\tau:T\ra \mathcal{QR}M$ is a morphism, and so the diagram  $T\s{\tau}\ra \mathcal{QR}M\ra \mathcal{R}M\leftarrow M$ is a complete $\A$ resolution.

$(6)\Ra(4)$. By hypothesis, there is a complete $\A$ resolution $T\s{\tau}\rightarrow \mathcal{QR}M\rightarrow\mathcal{R}M\leftarrow M$ such that $\tau_{\geqslant n}: T_{\geqslant n}\rightarrow (\mathcal{QR}M)_{\geqslant n}$ is an isomorphism of complexes. Hence $\textrm{C}_{n}$($\mathcal{QR}M$) $\cong$ $\textrm{C}_{n}$($T$) and \textrm{H}$_{i}(T)$ $\cong$ \textrm{H}$_{i}(\mathcal{QR}M)=0$ for all $i\geqslant n$. So $\textrm{C}_{n}$($\mathcal{QR}M$) $\in {\GA}$ and sup$M$ = sup$\mathcal{QR}M$ $\leqslant n$.
By \cite[1.6, p.134]{AFHd}, there exists a $dg$-projective complex $P$ with $P\simeq M$. Thus $\mathcal{QR}M\simeq P$, and hence there is a quasi-isomorphism $\varphi:P\ra \mathcal{QR}M$ by \cite[1.4.P, p.132]{AFHd}.

If $\varphi$ is surjective, then there is an exact sequence $0\ra K\ra P\s{\varphi}\ra \mathcal{QR}M\ra 0$ of complexes with $K$ an exact complex. Since both $P$ and $\mathcal{QR}M$ are $dg$-$\A$ complexes, $K$ is a $dg$-$\A$ complex by Lemma \ref{lem:3.3}. Note that $K$ is an exact complex. Then $K\in{\AT}$ and $\textrm{C}_{n}(K)\in{\A}$. Thus we have an exact sequence $0\ra \textrm{C}_{n}(K) \ra \textrm{C}_{n}(P)\ra \textrm{C}_{n}(\mathcal{QR}M)\ra 0 $ with $\textrm{C}_{n}(K)\in{\A}$  and $\textrm{C}_{n}(\mathcal{QR}M)\in{\GA}$, and so $\textrm{C}_{n}(P)\in{\GA}$ by Lemma \ref{prop:2.3}(1).

Suppose that $\varphi$ is not surjective. By Lemma \ref{lem:3.3}, there is a special $\widetilde{\A}$ precover $G\ra \mathcal{QR}M$ of $\mathcal{QR}M$. Thus ${P\oplus G}\ra \mathcal{QR}M$ is a surjective quasi-isomorphism with $P\oplus G\in{dg\AT}$ such that $\textrm{C}_{n}(P\oplus G)\cong {\textrm{C}_{n}(P)\oplus \textrm{C}_{n}(G)}$. Note that $\textrm{C}_{n}(P\oplus G)\in{\GA}$ by the proof above, and so $\textrm{C}_{n}(P)\in{\GA}$ by Lemma \ref{prop:2.3}(1), as desired.

To show the last claim, $\GAT\textrm{-dim}M=\textrm{sup}\{-\inf\R\Hom(M,X) \ | \ X\in{\AB} \}$ holds when $\GAT\textrm{-dim}M = -\infty$ by noting that $\GAT\textrm{-dim}M = -\infty$ if and only if $M$ is exact. By assumption, we assume that $\GAT\textrm{-dim}M = g< \infty$ with $g$ an integer. First we need to show that $\textrm{sup}\{-\inf\R\Hom(M,X) \ | \ X\in{\AB}  \}\leqslant g$.  By (3) and Lemma \ref{prop:2.3}(1), there exists a $dg$-projective complex $P$ such that $P\simeq M$ and $\textrm{C}_{n}(P)\in{\GA}$ for all $n\geqslant g$. For each $i\geqslant1$ and every $X\in{\AB}$, we get that

\vspace{2mm}
\hspace{23.5mm}$\h_{-g-i}(\R \Hom(M,X)) = \h_{-g-i}(\Hom(P,X))$

\vspace{1mm}
\hspace{64mm}$=\h_{-1}(\Hom(\Sigma^{-g-i+1}P_{\geqslant g+i-1},X))$

\vspace{1mm}
\hspace{64mm}$=\h_{-1}(\R\Hom(\textrm{C}_{g+i-1}(P),X))$

\vspace{1mm}
\hspace{64mm}$=\Ext^{1}(\textrm{C}_{g+i-1}(P),X)$

\vspace{1mm}
\hspace{64mm}$=0$.

\vspace{2mm}
\noindent This implies that $-\inf\R\Hom$ $(M,X)$ $\leqslant g$ for all $X\in{\AB}$, as desired.

Next we need to show that $\textrm{sup}\{-\inf\R\Hom(M,X) \ | \ \textrm{X}\in{\AB}  \} \geqslant g$. Suppose on the contrary that
$\textrm{sup}\{-\inf\R\Hom(M,X) \ | \ \textrm{X}\in{\AB} \} < g$. Since $\GAT\textrm{-dim}M = g$, there exists a complete $\A$ resolution $T\s{\tau}\ra \mathcal{QR}M\ra \mathcal{R}M\leftarrow M$ with each $T_i\in{\AB}$ such that $\textrm{C}_g(T)=\textrm{C}_g(\mathcal{QR}M)$ and $\tau_{i}$ $=$ id$_{T_{i}}$ for all $i\geqslant g$. Let $\lambda:(\mathcal{QR}M)_{g}\ra \textrm{C}_{g}(\mathcal{QR}M)$ be the natural map. Then there exist a monomorphism $l:\textrm{C}_{g}(\mathcal{QR}M)\ra T_{g-1}$ such that $\partial_{g}^{T} =l \lambda $ and a morphism $\mu:\textrm{C}_{g}(\mathcal{QR}M)\ra (\mathcal{QR}M)_{g-1}$ such that $\partial_{g}^{\mathcal{QR}M} =\mu \lambda $. By assumption, we get $\h_{-g}(\R\Hom(M,T_{g-1}))=0$. Thus $\h_{-g}(\Hom(\mathcal{QR}M,T_{g-1}))=0$ by Lemma \ref{lem:3.4}, and so we have the following exact sequence
$$\Hom((\mathcal{QR}M)_{g-1},T_{g-1})\ra\Hom((\mathcal{QR}M)_{g},T_{g-1})\ra\Hom((\mathcal{QR}M)_{g+1},T_{g-1}).$$
One easily checks that $\Hom(\mu,T_{g-1}):\Hom((\mathcal{QR}M)_{g-1},T_{g-1})\ra \Hom(\textrm{C}_{g}(\mathcal{QR}M), T_{g-1})$ is an epimorphism. Then there exists a map $\beta:(\mathcal{QR}M)_{g-1}\ra T_{g-1}$ such that $l=\beta \mu$. Since $l$ is monic, so is $\mu$. Note that $l: \textrm{C}_{g}(T)\ra T_{g-1} $ is a monic $\AB$-preenvelope. It follows that $\mu$ is a monic $\AB$-preenvelope. Thus $\textrm{C}_{g-1}(\mathcal{QR}M)\in{^{\perp_{1}}}(\AB)$, and so $\textrm{C}_{g-1}(A)\in{\GA}$ by Lemma \ref{prop:2.3}(2). This is a contradiction.
\end{proof}

By \cite[Definition 3.1]{YD}, the $\mathcal{A}$ dimension of $M$, denoted by $\mathcal{A}$-dim$M$, is defined as $\mathcal{A}$-dim$M$ =
$\inf$\{sup\{$i$ $|$ $A_{-i}\neq0$\} $|$ $M\simeq A$ with $A\in{dg\widetilde{\mathcal{A}}}$\}.
\begin{cor}\label{cor:4.3} Let $R$ be a ring and $M$ a complex. Then there is an inequality
{\rm\begin{center}{$\GAT\textrm{-dim}M\leqslant \mathcal{A}\textrm{-dim}M,$}
\end{center}}
\noindent and the equality holds if {\rm$\mathcal{A}\textrm{-dim}M<\infty$}.
\end{cor}
\begin{proof}Set $\mathcal{A}\textrm{-dim}M= m$ and $\GAT\textrm{-dim}M=n$. There is nothing to prove if $m=\infty$. We may assume that $m$ is finite. It follows from Proposition \ref{thm:3.9} and \cite[Theorem 3.3]{YD} that $n\leqslant m$.

Suppose $n<m$. Choose a $dg$-projective complex $P$ such that $P\simeq M$. Note that sup$M$ $\leqslant n$ and $\textrm{C}_n(P)\in{\GA}$ by Proposition \ref{thm:3.9}. Then there exists an exact sequence of $R$-modules
$$0\ra \textrm{C}_m(P)\ra P_{m-1}\ra \cdots\ra P_n\ra \textrm{C}_n(P)\ra0.$$
Applying \cite[Theorem 3.3]{YD} again, we get that $\textrm{C}_m(P)\in{\A}$. Let $L=\coker(\textrm{C}_m(P)\ra P_{m-1})$. Then $L\in{\GAT}$ by Proposition \ref{prop:2.3}(1). Since ($\A$, $\B$) is complete, there exists an exact sequence $0\rightarrow \textrm{C}_m(P)\rightarrow B\rightarrow A\rightarrow0$ with $B\in{\AB}$ and $A\in{\A}$. By pushout, we have the following commutative diagram with exact rows and columns:
$$\xymatrix@C=20pt@R=20pt{&0\ar[d]&0\ar[d]&&\\
0\ar[r]&\textrm{C}_m(P)\ar[r]\ar[d]&P_{m-1}\ar[r]\ar[d]&L
\ar[r]\ar@{=}[d]& 0\\
0\ar[r]&B\ar[d]\ar[r]&U\ar[d]\ar[r]&L
\ar[r]& 0\\
&A\ar@{=}[r]\ar[d]&A\ar[d]&&\\
&0&0.&&\\
}$$
Since $P_{m-1}$ and $A$ belong to $\A$, so is $U$. Note that $L\in{\GA}$ and $B\in{\AB}$, it follows that the middle row in the above diagram is split. So $L$ belongs to $\A$. We proceed in this manner to get that $\textrm{C}_n(P)$ belongs to $\A$. Now \cite[Theorem 3.3]{YD} implies $m\leqslant n$, which contradicts our assumption. So $m=n$, as desired.
\end{proof}

\begin{cor}\label{corollary:3.11}For every family of complexes $\{M_{i}\}_{i\in{I}}$ one has
{\rm $$\GAT\textrm{-dim}\bigoplus_{i\in{I}} M_{i}=\sup_{i\in{I}}\{\GAT\textrm{-dim}M_{i}\}.$$}
\end{cor}
\begin{proof} The proof is similar to that of \cite[Corollay 3.5]{vg}.
\end{proof}

\begin{fact}\label{fact:2.5}{\rm  Let $M$ and $N$ be two complexes. According to Lemma \ref{lem:3.3},
there are two cofibrant-fibrant resolutions
$\mathcal{QR}M \ra  \mathcal{R}M \leftarrow M $  and $\mathcal{QR}N \ra  \mathcal{R}N \leftarrow N$ of $M$ and $N$, respectively. A homomorphism $\beta\in{\Hom(\mathcal{QR}M,\mathcal{QR}N)}$ is \emph{bounded above}
  if $\beta_i=0$ for all $i\gg0$. The subset $\uHom(\mathcal{QR}M,\mathcal{QR}N)$ of $\Hom(\mathcal{QR}M,\mathcal{QR}N)$, consisting of all bounded above homomorphisms, is a subcomplex with components $$\uHom(\mathcal{QR}M,\mathcal{QR}N)_{n}=\{(\varphi_i)\in{\Hom(\mathcal{QR}M,\mathcal{QR}N)_n \ | \ \varphi_{i} \ \textrm{=} \ \textrm{0} \ \textrm{for} \ \textrm{all} \ i \gg 0}\}.$$
We set $$\wHom(\mathcal{QR}M,\mathcal{QR}N)=\Hom(\mathcal{QR}M,\mathcal{QR}N)/{\uHom(\mathcal{QR}M,\mathcal{QR}N)}.$$
By \cite[Definition 3.7]{HuDing}, the $n$th \emph{Tate-Vogel cohomology group}, denoted by $\wtE_{\AAB}^n(M,N)$, is defined as
$$\wtE_{\AAB}^n(M,N)=\textrm{H}_{-n}(\wHom(\mathcal{QR}M,\mathcal{QR}N)).$$}
\end{fact}

\begin{rem} \label{rem:4.1}{\rm Let $M$ and $N$ be two complexes. According to Fact \ref{fact:2.5}, there exists an exact sequence of complexes $$0\ra\uHom(\mathcal{QR}M,\mathcal{QR}N)\ra \Hom(\mathcal{QR}M,\mathcal{QR}N) \ra \wHom(\mathcal{QR}M,\mathcal{QR}N)\ra 0.$$
It follows from Lemma \ref{lem:3.4} that $\Ext^{n}(M,N)=\h_{-n}(\Hom(\mathcal{QR}M,\mathcal{QR}N))$ for each integer $n$. Thus we have a long exact sequence
$$\xymatrix@C=0.93cm{
  \cdots \ar[r]& \Yxt_{\AAB}^n(M,N) \ar[r]^{\varepsilon_R^n(M,N)} & \Ext^n(M,N) \ar[r]^{\widetilde{\varepsilon}_R^n(M,N)} & \wtE_{\AAB}^n(M,N)\ar[r] &\Yxt_{\AAB}^{n+1}(M,N) \ar[r]&\cdots,}$$
where $\Yxt_{\AAB}^n(M,N)=\h_{-n}(\uHom(\mathcal{QR}M,\mathcal{QR}N))$ for each integer $n$.

By \cite[Lemma 3.3]{HuDing}, one can see that $\Yxt_{\AAB}^n(-,-)$ is a cohomological functor for each integer $n$, independent of the choice of cofibrant replacements and fibrant replacements.}
\end{rem}

We are now in a position to prove Theorem \ref{theorem:3.6}, which gives general techniques for computing cohomologies $\Yxt_{\AAB}^n(M,N)$ and $\wte_{\AAB}^n(M,N)$ whenever $M$ is a complex with finite Gorenstein $\A$ dimension.

\begin{para}\label{4.3} {\bf Proof of Theorem \ref{theorem:3.6}.}   By Proposition \ref{thm:3.9}, there exists a split complete $\A$ resolution $T\s{\tau}\ra\mathcal{QR}M\ra \mathcal{R}M\leftarrow M$ of $M$ such that each $\tau_{i}$ is a split epimorphism and $\tau_{i}$ $=$ id$_{T_{i}}$ for all $i\gg0$. Let $L_M =\ker(\tau)$. Then the exact sequence $0\ra L_M\ra T\ra \mathcal{QR}M\ra 0$ of complexes is split in  each degree. Let $\mathcal{QR}N\ra \mathcal{R}N\leftarrow N$ be a cofibrant-fibrant resolution of $N$. Applying the functor $\Hom(-,\mathcal{QR}N)$
 to the  exact sequence above, we have the following commutative diagram with exact rows and columns:
$$\xymatrix@C=15pt@R=15pt{
  &0 \ar[d]& 0 \ar[d] & 0 \ar[d] & \\
  0\ar[r]& \uHom(\mathcal{QR}M,\mathcal{QR}N) \ar[r] \ar[d] & \Hom(\mathcal{QR}M,\mathcal{QR}N)  \ar[r] \ar[d] & \wHom(\mathcal{QR}M,\mathcal{QR}N)  \ar[d] \ar[r] &0\\
  0\ar[r]& \uHom(T,\mathcal{QR}N) \ar[r] \ar[d] & \Hom(T,\mathcal{QR}N)  \ar[r] \ar[d] & \wHom(T,\mathcal{QR}N)  \ar[d] \ar[r] &0\\
  0\ar[r]& \uHom(L_M,\mathcal{QR}N) \ar[r] \ar[d] & \Hom(L_M,\mathcal{QR}N)  \ar[r] \ar[d] & \wHom(L_M,\mathcal{QR}N)  \ar[d] \ar[r] &0\\
  &0 & 0  &\ 0.  &}$$
Since $L_M\in{\C^{+}(R)}$, $\uHom(L_M,\mathcal{QR}N)=\Hom(L_M,\mathcal{QR}N)$. For each integer $n\in{\mathbb{Z}}$, we have $\h_n(\wHom(\mathcal{QR}M,\mathcal{QR}N))\cong\h_n(\wHom(T,\mathcal{QR}N))$. Note that $\mathcal{QR}N\in{dg\AT \cap dg\BT}$ and $T$ is an exact complex such that $\Hom(T,X)$ is exact for all $X\in{\AB}$. Then $\uHom(T,\mathcal{QR}N)$ is exact. Thus $\h_n(\Hom(T,\mathcal{QR}N))\cong\h_n(\wHom(T,\mathcal{QR}N))$ for each integer $n\in{\mathbb{Z}}$, and hence $$\h_n(\wHom(\mathcal{QR}M,\mathcal{QR}N))\cong\h_n(\wHom(T,\mathcal{QR}N))\cong\h_n(\Hom(T,\mathcal{QR}N))$$ for each integer $n\in{\mathbb{Z}}$.
So $\wtE_{\AAB}^n(M,N)\cong \h_{-n}(\Hom(T,\mathcal{QR}N))$ for each integer $n\in{\mathbb{Z}}$.

By the exactness of the left column in the above diagram, we have the following exact sequence
\begin{center}{$\xymatrix@C=30pt@R=30pt{\cdots\ar[r]& \h_{n}(\uHom(\mathcal{QR}M,\mathcal{QR}N))\ar[r]& \h_{n}(\uHom(T,\mathcal{QR}N))
\ar[r]& \h_{n}(\uHom(L_M,\mathcal{QR}N))}$

\hspace{-23mm}$\xymatrix@C=30pt@R=30pt{ \ar[r]&\h_{n-1}(\uHom(\mathcal{QR}M,\mathcal{QR}N))\ar[r]& \h_{n-1}(\uHom(T,\mathcal{QR}N))
\ar[r]&\cdots}$.}
\end{center}
Since $\uHom(T,\mathcal{QR}N)$ is exact by the proof above, we get that $$\h_{n-1}(\uHom(\mathcal{QR}M,\mathcal{QR}N))\cong\h_{n}(\uHom(L_M,\mathcal{QR}N))$$ for all integers $n\in{\mathbb{Z}}$. So
$\Yxt_{\AAB}^{n+1}(M,N)\cong\h_{-n}(\Hom(L_M,\mathcal{QR}N))$ for all integers $n\in{\mathbb{Z}}$.

Now assume that $N$ is a bounded below $dg$-$\mathcal{B}$ complex. Note that there is an exact sequence
$0\ra K_N\ra \mathcal{Q}N \ra N\ra 0$ of complexes with
$K_N\in{\BT\cap{\C}^{-}(R)}$ and $\mathcal{Q}N\in{dg\AT \cap {\C}^{-}(R)}$ by \cite[Lemma 3.11]{HuDing}.
Hence both $\Hom(T,K_N)$ and $\Hom(L_M,K_N)$ are exact by \cite[Lemma 2.4]{cfh}.
It is easy to check that the sequences $$0\ra \Hom(T,K_N)\ra  \Hom(T,\mathcal{Q}N)\ra \Hom(T,N)\ra 0$$ and
$$0\ra \Hom(L_M,K_N)\ra  \Hom(L_M,\mathcal{Q}N)\ra \Hom(L_M,N)\ra 0$$ are exact. Thus for any integer $n$, we have $$\h_{n}(\Hom(T,\mathcal{Q}N))\cong \h_n (\Hom(T,N)),\ \h_{n}(\Hom(L_M,\mathcal{Q}N))\cong \h_n (\Hom(L_M,N)).$$ So $\wtE_{\AAB}^n(M,N)\cong \h_{-n}(\Hom(T,N)) \ \textrm{and} \ \Yxt_{\AAB}^{n+1}(M,N)\cong\h_{-n}(\Hom(L_M,N))$ for all integers $n\in{\mathbb{Z}}$. This completes the proof.\hfill$\Box$
\end{para}

\begin{cor}\label{cor:5.5} If $M$ is a complex of finite Gorenstein $\A$ dimension, then we have
{\rm $$\GAT\textrm{-dim}M={\rm sup}\{n\in{\mathbb{Z}} \ | \
\Yxt_{\AAB}^{n}(M,N)\neq 0  \ \textrm{for} \ \textrm{some} \ N\in{\AB} \}.$$}
\end{cor}
\begin{proof} Let $N\in{\AB}$. Then $\wtE_{\AAB}^n(M,N)=0$ for all $n\in{\mathbb{Z}}$ by Theorem \ref{theorem:3.6}. It follows from Fact \ref{fact:2.5}(2) that $\Yxt_{\AAB}^{n}(M,N)\cong\textrm{Ext}^{n}_{R}(M,N)$ for all
$n\in{\mathbb{Z}}$ and all $N\in{\AB}$. So the result is clear by Proposition \ref{thm:3.9}.
\end{proof}

\section{\bf Relative cohomology groups for complexes with finite Gorenstein $\A$ dimension}\label{Relative-homology}

The goal of this section is to study relative cohomology groups for complexes with finite
Gorenstein $\A$ dimension and prove Theorem \ref{thm:1.4'}. To this end, we start with the following easy observations.

\begin{lem}\label{lem:5.1} Let $f:M\ra B$ be a special $\B$ preenvelope of an $R$-module $M$. Then
{\rm$\mathcal{A}\textrm{-dim}M=\mathcal{A}\textrm{-dim}B$} and {\rm $\GAT\textrm{-dim}M=\GAT\textrm{-dim}B$}.
\end{lem}
\begin{proof} The proof is straightforward by noting that $0\ra M\ra B\ra A\ra 0$ is an exact sequence of $R$-modules with $A\in{\A}$.
\end{proof}

\begin{lem}\label{prop:5.2}Let $M$ be an $R$-module with finite Gorenstein $\A$ dimension and $f:M\ra B$ a special $\B$ preenvelope of $M$. Then $B$ has a split complete $\A$ resolution $T\s{\tau}\rightarrow \mathcal{Q}B\rightarrow B\s{{\rm id}}\leftarrow B$. Hence there exists a degreewise split exact sequence of complexes $$0\ra \Sigma^{-1}X\ra \widetilde{T}\ra \mathcal{Q}B\ra0$$
with $\widetilde{T}=(T_{\geqslant0})^{+}$ such that $X$ is a proper $\GA$-resolution of $B$.
\end{lem}
\begin{proof}By Proposition \ref{thm:3.9} and Lemma \ref{lem:5.1}, $B$ has a split complete $\A$ resolution $T\s{\tau}\rightarrow \mathcal{Q}B\rightarrow B\s{{\rm id}}\leftarrow B$. Thus there is a nonnegative integer $n$ such that $\tau_{i}$ is bijective for all $i\geqslant n$.
We set $\widetilde{T}=(T_{\geqslant0})^{+}$, that is
 \begin{center}{$\widetilde{T}_i=\left\{\begin{array}{llll}
T_{i}  &  \mbox{if $i\geqslant 0$;}\\
\textrm{C}_{0}(T)  &  \mbox{if $i =-1$;}\\
0  & \mbox{if $i<-1$,}
\end{array}
\right.$ \quad and \quad $\partial^{\widetilde{T}}_i=\left\{\begin{array}{llll}
\partial^{T}_{i}  &  \mbox{if $i>0$;}\\
\pi  &  \mbox{if $i =0$;}\\
0  & \mbox{if $i<0$,}
\end{array}
\right.$}
\end{center}
\noindent where $\pi:T_{0}\ra \textrm{C}_{0}(T)$ is the natural map. Let $\widetilde{\beta}:\widetilde{T}\ra \mathcal{Q}B$ be a morphism such that $\widetilde{\beta}_{i}=\tau_i$ for all $i\geqslant 0$ and $\widetilde{\beta}_{i}=0$ for all $i<0$. Let $X=\Sigma\ker(\widetilde{\beta})$. Note that $\ker(\tau)$ is a complex with each entry in $\mathcal{A}\cap\mathcal{B}$. Thus
$X_0=\textrm{C}_{0}(T)\in{{\GA}}$, $X_i\in{\mathcal{A}\cap\mathcal{B}}$ for $1\leqslant i\leqslant n-1$, and $X_i=0$ for $i\geqslant{n+1}$ and $i\leqslant{-1}$. Hence $X$ is a proper $\GA$-resolution of $B$. So we have the following exact sequence of complexes:
$$0\ra \Sigma^{-1}X\ra \widetilde{T}\ra \mathcal{Q}B\ra0$$
with $\widetilde{T}=(T_{\geqslant0})^{+}$.
\end{proof}

The following is the key result to prove Theorem \ref{thm:1.4'}.

\begin{prop}\label{thm:1.4} Let $M$ be an $R$-module in $\mathcal{B}$ with finite Gorenstein $\A$ dimension. For any $R$-module $N$ in $\B$, we have the following isomorphisms:
\begin{enumerate}
\item {\rm $\oext^{1}_{\GA}(M,N)\cong\ker(\widetilde{\varepsilon}_R^1(M,N))$;
\item {\rm $\oext^{n}_{\GA}(M,N)\cong\Yxt_{\AAB}^{n}(M,N)$}} for any integer $n>1$.
\end{enumerate}
\end{prop}
\begin{proof}Note that $M$ is an $R$-module in $\B$ with finite Gorenstein $\A$ dimension. Then $M$ has a complete $\A$ resolution $T\s{\tau}\rightarrow \mathcal{Q}M\rightarrow M\s{{\rm id}}\leftarrow M$ by Proposition \ref{thm:3.9}. It follows from Lemma \ref{prop:5.2} that there exists a
degreewise split exact sequence of complexes $$\mathbb{L}:0\ra \Sigma^{-1}X\ra \widetilde{T}\ra \mathcal{Q}M\ra0$$
with $\widetilde{T}=(T_{\geqslant0})^{+}$ such that $X$ is a proper $\GA$-resolution of $M$.

(1) Let $N$ be an $R$-module in $\B$. Applying $\Hom(-,N)$ to the exact sequence $\mathbb{L}$ gives rise to the following exact sequence of complexes:
$$0\ra\Hom(\mathcal{Q}M,N)\ra\Hom(\widetilde{T},N)\ra \Hom( \Sigma^{-1}X,N)\ra0.$$
By the long exact sequence theorem, we have the following exact sequence:
$$\xymatrix@C=10pt@R=10pt{\h_0(\Hom(\widetilde{T},N)) \ar[r] & \h_0(\Hom(\Sigma^{-1}X,N)) \ar[r] &
  \h_{-1}(\Hom(\mathcal{Q}M,N))\ar[r]& \h_{-1}(\Hom(\widetilde{T},N)).}$$
It follows from Lemma \ref{lem:3.4} and Theorem \ref{theorem:3.6} that the following sequence
$$\xymatrix@C=0.93cm{0 \ar[r] & \h_0(\Hom(\Sigma^{-1}X,N))  \ar[r] & \Ext^{1}(M,N) \ar[r]^{\varepsilon_R^1(M,N)}
  &\wtE_R^{1}(M,N)}$$ is exact. Since $X$ is a proper $\GA$-resolution of $M$, ${\rm Ext}_{\GA}^{1}(M,N)\cong\h_{-1}(\Hom(X,N))$ by Definition \ref{df:2.0}. Note that
$\h_{-1}(\Hom(X,N))\cong \h_0(\Hom(\Sigma^{-1}X,N))$. So ${\rm Ext}_{\GA}^{1}(M,N)\cong\ker(\widetilde{\varepsilon}_R^1(M,N))$, as desired.

(2) Let $n$ be any integer with $n>1$. Note that $\Yxt_{\AAB}^n(M,N)\cong\h_{1-n}(\Hom(L_{M},N))$ by Theorem \ref{theorem:3.6}. By Lemma \ref{prop:5.2},
we have $$\h_{1-n}(\Hom(L_{M},N))=\h_{1-n}(\Hom(\Sigma^{-1}X,N)).$$ Then $\Yxt_{\AAB}^n(M,N)\cong\h_{1-n}(\Hom(\Sigma^{-1}X,N))$. It follows from Definition \ref{df:2.0} that ${\rm Ext}_{\GA}^{n}(M,N)\cong\h_{-n}(\Hom(X,N))$ since $X$ is a proper $\GA$-resolution of $M$. Note that $\h_{-n}(\Hom(X,N))\cong \h_{1-n}(\Hom(\Sigma^{-1}X,N))$. So $\oext^{n}_{\GA}(M,N)\cong\Yxt_{\AAB}^n(M,N)$. This completes the proof.
\end{proof}

\begin{lem}\label{lem:5.1''} Let  $0\ra M\ra B\ra A\ra 0$ be an exact sequence of $R$-modules with $B\in{\B}$ and $A\in{\A}$. For each $R$-module $N$ in $\B$, we have the following commutative diagram with exact rows
$$\xymatrix@C=1.5cm{
  \cdots\ar[r]&\Yxt_{\AAB}^i(B,N) \ar[d]^{f^i} \ar[r] & \Ext^i(B,N) \ar[d]^{g^i} \ar[r]^{\widetilde{\varepsilon}_R^i(B,N)} & \wte_{\AAB}^i(B,N) \ar[d]^{h^i}\ar[r] &\cdots\\
   \cdots\ar[r]&\Yxt_{\AAB}^i(M,N)  \ar[r] & \Ext^i(M,N)  \ar[r]^{\widetilde{\varepsilon}_R^i(M,N)} & \wte_{\AAB}^i(M,N)  \ar[r]&\cdots}$$
satisfying that
\begin{enumerate}
\item{\rm$g^i:\Ext^i(B,N)\ra \Ext^i(M,N)$} is an isomorphism for any $i\geqslant1$.

\item  {\rm$h^i:\wte_{\AAB}^i(B,N)\ra \wte_{\AAB}^i(M,N)$} is an isomorphism for each $i\in{\mathbb{Z}}$.
\end{enumerate}
\end{lem}
\begin{proof}Since $A\in{\A}$, there is an exact sequence $0\ra A\ra X_{0}\ra X_{-1}\ra \cdots$ of $R$-modules such that $X_i\in{\AB}$ for $i\leqslant0$ and
$\ker(X_{i}\ra X_{i-1})\in{\A}$ for $i\leqslant0$. Let $\mathcal{R}A$ be the complex $ 0\ra X_{0}\ra X_{-1}\ra \cdots$. Then $A\ra \mathcal{R}A$ is a fibrant replacement with $\mathcal{R}B$ in $dg\AbT \cap dg\BBT$. Hence $\mathcal{QR}A=\mathcal{R}A$.
By \cite[Lemma 3.12]{HuDing}, we have the following commutative diagram with exact rows such that the columns are fibrant
replacements:
$$\xymatrix@C=1.5cm{
  0\ar[r]&  M \ar[d]\ar[r] &  B \ar[d]\ar[r] &  A \ar[d]\ar[r] &0\\
   0\ar[r]&\mathcal{R}M \ar[r] & \mathcal{R}B \ar[r] &\mathcal{R}A \ar[r]&0.}$$
Note that there exists an exact sequence $0\ra L\ra \mathcal{QR}B\ra \mathcal{R}B\ra 0$ of complexes with $L$ an $\B$ complex. Consider the following pullback diagram:
$$\xymatrix{&0\ar[d]&0\ar[d]&&\\
&L\ar[d]\ar@{=}[r]&L\ar[d]&&\\
0\ar[r]&W\ar[r]\ar[d]&\mathcal{QR}B\ar[r]\ar[d]&\mathcal{R}A
\ar[r]\ar@{=}[d]& 0\\
0\ar[r]&\mathcal{R}M\ar[r]\ar[d]&\mathcal{R}B\ar[d]\ar[r]&\mathcal{R}A
\ar[r]& 0\\
&0& \ 0.&&\\
}$$
Since $\mathcal{R}A$ and  $\mathcal{QR}B$ are in $dg\widetilde{\mathcal{A}}$, so is $W$ by Lemma \ref{lem:3.3}. Let $\mathcal{QR}M=W$. Then $\mathcal{QR}M\ra\mathcal{R}M$ is a cofibrant
replacement. Thus we have the following commutative diagram with exact rows such that the columns are cofibrant-fibrant
resolutions:
$$\xymatrix@C=1.5cm{
  0\ar[r]& M \ar[d] \ar[r] & B \ar[r]\ar[d] &  A \ar[d]\ar[r]&0\\
   0\ar[r]&\mathcal{R}M \ar[r] & \mathcal{R}B  \ar[r] & \mathcal{R}A \ar[r]&0\\
    0\ar[r]&\mathcal{QR}M\ar[r]\ar[u] & \mathcal{QR}B \ar[r] \ar[u]& \mathcal{QR}A \ar[r]\ar@{=}[u]&0.}$$

Let $N$ be an $R$-module in $\mathcal{B}$. Choose any cofibrant-fibrant resolution $N\rightarrow \mathcal{R}N \leftarrow \mathcal{QR}N$ of $N$. Then we have the following commutative diagram with exact rows and columns:
$$\xymatrix@C=15pt@R=15pt{
  &0 \ar[d]& 0 \ar[d] & 0 \ar[d] & \\
  0\ar[r]& \dHom(\mathcal{QR}A,\mathcal{QR}N) \ar[r] \ar[d] & \Hom(\mathcal{QR}A,\mathcal{QR}N)  \ar[r] \ar[d] & \lHom(\mathcal{QR}A,\mathcal{QR}N)  \ar[d] \ar[r] &0\\
  0\ar[r]& \dHom(\mathcal{QR}B,\mathcal{QR}N) \ar[r] \ar[d] & \Hom(\mathcal{QR}B,\mathcal{QR}N)  \ar[r] \ar[d] & \lHom(\mathcal{QR}B,\mathcal{QR}N)  \ar[d] \ar[r] &0\\
  0\ar[r]& \dHom(\mathcal{QR}M,\mathcal{QR}N) \ar[r] \ar[d] & \Hom(\mathcal{QR}M,\mathcal{QR}N)  \ar[r] \ar[d] & \lHom(\mathcal{QR}M,\mathcal{QR}N)  \ar[d] \ar[r] &0\\
  &0 & 0  &\ 0.  &}$$
Since $\mathcal{QR}A\in{\mathbb{C}^{+}(R)}$, $\dHom(\mathcal{QR}A,\mathcal{QR}N)=\Hom(\mathcal{QR}A,\mathcal{QR}N)$. For each integer $n\in{\mathbb{Z}}$, we have $\h_n(\lHom(\mathcal{QR}B,\mathcal{QR}N))\cong\h_n(\lHom(\mathcal{QR}M,\mathcal{QR}N))$. It follows from Lemma \ref{lem:3.4} that  $\Ext^{i}(A,N)\cong\h_{-i}(\Hom(\mathcal{QR}A,\mathcal{QR}N))$ for each integer $i\in{\mathbb{Z}}$. Since $A\in{\A}$ and $N\in{\B}$, $\h_i(\Hom(\mathcal{QR}A,\mathcal{QR}N))=0$ for all $i\leqslant-1$. Thus $\h_i(\Hom(\mathcal{QR}B,\mathcal{QR}N))\cong\h_i(\Hom(\mathcal{QR}M,\mathcal{QR}N))$ for all $i\leqslant-1$. Applying the long exact sequence theorem to the commutative diagram above, we have
the desired commutative diagram in Lemma \ref{lem:5.1''}.
\end{proof}

\begin{fact}\label{fact:5.4}{\rm Let $M$ be an $R$-module in $\B$ with finite Gorenstein $\A$ dimension. Then $M$ has a proper $\GA$-resolution $X\ra M$ by Lemma \ref{prop:5.2}. Note that the class $\mathcal{A}$ is pecovering by hypothesis. Choose a proper $\A$-resolution  $F\ra M$ and a morphism $\gamma:F\ra X$ lifting the identity on $M$. For each $R$-module $N$ in $\B$, we have $\Hom(F,N)\cong \R\Hom(M,N)$ by Lemma \ref{lem:3.4} and the morphism of complexes
$$\Hom(\gamma,N):\Hom(X,N)\ra \Hom(F,N) $$
induces a natural homomorphism of abelian groups $$\varepsilon^{n}(M,N):\textrm{Ext}^{n}_{\GA}(M,N)\ra \textrm{Ext}^{n}_{R}(M,N)$$
for every $n\in \mathbb{Z}$. The groups and the maps defined above do not depend on the choices of resolutions and liftings by
\cite[Proposition 2.2]{Hgd}.}
\end{fact}

\begin{lem}\label{lem:5.1'''} Let $M$ be an $R$-module with finite Gorenstein $\A$ dimension and $N$ an $R$-module in $\B$. For each integer $i$ with $i\geqslant1$, we have the following commutative diagram such that the downward arrows are isomorphisms
$$\xymatrix@C=1.5cm{
  {\rm Ext}^{i}_{\GA}(B,N)\ar[d]^{\cong} \ar[r]^{\varepsilon^{i}(B,N)} & \Ext^i(B,N) \ar[d]^{\cong} \\
  {\rm Ext}^{i}_{\GA}(M,N) \ar[r]^{\varepsilon^{i}(M,N)} & \Ext^i(M,N),}$$
where $M\ra B$ is a special $\B$ preenvelope of $M$.
\end{lem}
\begin{proof}Let $f:M\ra B$ be a special $\B$ preenvelope of $M$. Then we have an exact sequence $0\ra M\ra B\ra A\ra 0$ of $R$-modules with $B\in{\B}$ and $A\in{\A}$. Since $\mathcal{GA}\textrm{-dim}M<\infty$, there is a nonnegative integer $n$ such that {\rm $\mathcal{GA}\textrm{-dim}M=\mathcal{GA}\textrm{-dim}B\leqslant n$}. By the proof of Lemma \ref{prop:5.2}, $B$ has a proper $\GA$-resolution $\beta:X\ra B$ such that $X_{0}\in{\GA}$, $X_{i}\in{\AB}$ for $1\leqslant i\leqslant n$ and $X_{i}=0$ for $i> n$. Consider the following pullback diagram:
$$\xymatrix{&0\ar[d]&0\ar[d]&&\\
&\textrm{C}_{1}(X)\ar[d]\ar@{=}[r]&\textrm{C}_{1}(X)\ar[d]&&\\
0\ar[r]&L\ar[r]^{\lambda}\ar[d]^{\alpha}&X_{0}\ar[r]\ar[d]^{\beta_{0}}&A
\ar[r]\ar@{=}[d]& 0\\
0\ar[r]&M\ar[r]^{f}\ar[d]&B\ar[d]\ar[r]&\A
\ar[r]& 0\\
&0& \ 0.&&\\
}$$
Note that $A\in{\A}$ and $X_{0}\in{\GA}$. Then $L\in{\GA}$ by Lemma \ref{prop:2.3}(1). Assume that $Y$ is a complex such that $Y_{0}=L$, $Y_{i}=X_{i}$ for  $1\leqslant i\leqslant n$ and $Y_{i}=0$ for all $i> n$. Thus $\eta:Y\ra M$ is a proper $\GA$-resolution of $M$, where $\eta_{0}=\alpha$ and $\eta_{i}=0$ for $i\neq0$. Let $K=\cdots\rightarrow 0\rightarrow A\rightarrow 0\rightarrow\cdots$ with $A$ in the 0th position and 0 in the other positions. It is easy to see that $0\ra Y\s{\gamma}\ra X\ra K\ra 0$ is an exact sequence of complexes, where $\gamma_{0}=\lambda$ and $\gamma_{i}={\rm id}_{X_{i}}$ for all $i\geqslant1$.

Note that $B\in{\B}$. Then $B$ has a proper $\A$-resolution $\beta':X'\ra B$ such that $X'_{i}\in{\AB}$ and ${\rm C}_{i}(X')\in{\B}$ for $i\geqslant0$. By the foregoing proof, $M$ has a proper $\A$-resolution $\eta':Y'\ra M$ with $Y'_{i}=X'_{i}$ for all $i\geqslant1$ such that the following diagram  $$\xymatrix{
0\ar[r]&Y'_{0}\ar[r]^{\mu}\ar[d]^{\eta'_{0}}&X'_{0}\ar[r]\ar[d]^{\beta'_{0}}&A
\ar[r]\ar@{=}[d]& 0\\
0\ar[r]&M\ar[r]^{f}&B\ar[r]&A
\ar[r]& 0\\
}$$ is commutative. Thus $0\ra Y'\s{\gamma'}\ra X'\ra K\ra0$ is an exact sequence of complexes, where $\gamma'_{0}=\mu$ and $\gamma'_{i}={\rm id}_{Y'_{i}}$ for all $i\geqslant1$. Since each $R$-module in $\A$ belongs to $\GA$, there exists $\varphi:X'\ra X$ such that $\beta\varphi=\beta'$. Note that $f\eta'_{0}=\beta'_{0}\mu=\beta_{0}\varphi_{0}\mu$. Using the pullback of homomorphisms $f$ and $\beta_{0}$, we have a morphism $\rho:Y'_{0}\ra Y_{0}$ such that $\lambda\rho=\varphi_{0}\mu$. Let $\psi:Y'\ra Y$ be a morphism such that $\psi_{0}=\rho$ and $\psi_{i}=\varphi_{i}$ for all $i\geqslant1$. Thus we have the following diagram of complexes with exact rows:
$$\xymatrix{
0\ar[r]&Y'\ar[r]^{\gamma'}\ar[d]^{\psi}&X'\ar[r]\ar[d]^{\varphi}&Y
\ar[r]\ar@{=}[d]& 0\\
0\ar[r]&Y\ar[r]^{\gamma}&X\ar[r] &K
\ar[r]& 0.\\
}$$
Let $N$ be an $R$-module in $\B$. Thus we have the following diagram of complexes:
$$\xymatrix{
0\ar[r]&\Hom(K,N)\ar[r]\ar@{=}[d]&\Hom(X,N)\ar[r]\ar[d]&\Hom(Y,N)
\ar[r]\ar[d]& 0\\
0\ar[r]&\Hom(K,N)\ar[r]&\Hom(X',N)\ar[r]&\Hom(Y',N)
\ar[r]& 0.\\
}$$
Note that $\textrm{H}_{i}(\Hom(K,N))=0$ for all $i\leqslant-1$. Applying the long exact sequence theorem to the commutative diagram above, we have
the desired commutative diagram in Lemma \ref{lem:5.1'''}.
\end{proof}

We now finish this section by giving the proof of Theorem \ref{thm:1.4'} as follows.
\begin{para}\label{4.10}{\bf Proof of Theorem \ref{thm:1.4'}.} Since $M$ is an $R$-module with  $\mathcal{GA}\textrm{-dim}M<\infty$, there exists an exact sequence $0\ra M\ra B\ra A\ra 0$ of $R$-modules with $B\in{\B}$ and $A\in{\A}$. Thus {\rm $\mathcal{GA}\textrm{-dim}M=\mathcal{GA}\textrm{-dim}B<\infty$} by Lemma \ref{lem:5.1}. Let $N$ be an $R$-module in $\B$. By Proposition \ref{thm:1.4}, Lemmas \ref{lem:5.1''} and \ref{lem:5.1'''}, we have $\oext^{1}_{\GA}(M,N)\cong\oext^{1}_{\GA}(B,N)\cong\ker(\widetilde{\varepsilon}_R^1(B,N))\cong \ker(\widetilde{\varepsilon}_R^1(M,N))$ and $\oext^{n}_{\GA}(M,N)\cong\oext^{n}_{\GA}(B,N)\cong\Yxt_{\AAB}^{n}(B,N)\cong\Yxt_{\AAB}^{n}(M,N)$ for any integer $n>1$. This completes the proof.
\hfill$\Box$
\end{para}


\section{\bf Comparisons between  Gorenstein $\A$ dimensions and $\A$ dimensions}\label{comparsion}
Our goal of this section is to investigate the relationships between  Gorenstein $\A$ dimensions and $\A$ dimensions for complexes. To this end, we start with the following exact sequence for modules with finite Gorenstein $\A$ dimension, connecting relative and Tate-Vogel cohomologies via a long exact sequence. We will refer to a sequence of this form as an Avramov-Martsinkovsky exact sequence. The sequence is similar to \cite[Theorem 7.1]{AM}.

\begin{prop}\label{thm:C'} Assume that $M$ is an $R$-module such that $\GAT$-dim$M$ $\leqslant g$ with $g\geqslant1$ an integer. For each $R$-module $N$ in $\B$, there
is a long exact sequence

\vspace{2mm}
\noindent\hspace{12mm}{\rm $\xymatrix@C=35pt@R=35pt{0\ar[r]& \oext^{1}_{\GA}(M,N)\ar[r]^{\varepsilon^{1}(M,N)}&\textrm{Ext}_{R}^{1}(M,N)
\ar[r]&\widetilde{\textrm{Ext}}^{1}_{\AAB}(M,N)\ar[r]&}$

\noindent \hspace{9mm}$\xymatrix@C=35pt@R=35pt{\cdots\ar[r]&\oext^{g}_{\GA}(M,N)\ar[r]^{\varepsilon^{g}(M,N)}&
\textrm{Ext}_{R}^{g}(M,N)\ar[r]&\widetilde{\textrm{Ext}}^{g}_{\AAB}(M,N)\ar[r]&0.}$}
\end{prop}
\begin{proof}
Let $M\ra B$ be a special $\B$ preenvelope of $M$. By Proposition \ref{thm:3.9} and Lemma \ref{lem:5.1}, there exists a split complete $\A$ resolution $T\s{\tau} \ra \mathcal{Q}B\ra B\s{{\rm id}}\leftarrow B$ of $B$ with each $\tau_{i}$ a split epimorphism such that $\tau_{i}$ $=$ {\rm id$_{T_{i}}$} for all $i\geqslant g$.
Thus $\ker(\tau)$ is a complex such that $(\ker(\tau))_i=0$ for all $i\geqslant g$, and hence $\h_{-g}(\Hom(\ker(\tau),N))=0$. Note that $\Yxt_{\AAB}^{g+1}(B,N)\cong\h_{-g}(\Hom(\ker(\tau),N))$ by Theorem \ref{theorem:3.6}. It follows that $\Yxt_{\AAB}^{g+1}(B,N)=0$.
By Remark \ref{rem:4.1}, we have the following exact sequence:
$$\mathbb{X}:\xymatrix@C=0.93cm{
  \cdots \ar[r]& \Yxt_{\AAB}^i(B,N) \ar[r] & \Ext^i(B,N) \ar[r] & \wte_{\AAB}^i(B,N)\ar[r] &\cdots.}$$
 Applying Theorem \ref{thm:1.4'} and Fact \ref{fact:5.4} to the exact sequence $\mathbb{X}$ above, we have the the following exact sequence:

\vspace{2mm}
\noindent\hspace{12mm}{\rm $\xymatrix@C=35pt@R=35pt{0\ar[r]& \oext^{1}_{\GA}(B,N)\ar[r]^{\varepsilon^{1}(B,N)}&\textrm{Ext}_{R}^{1}(B,N)
\ar[r]&\widetilde{\textrm{Ext}}^{1}_{\AAB}(B,N)\ar[r]&}$

\noindent \hspace{9mm}$\xymatrix@C=35pt@R=35pt{\cdots\ar[r]&\oext^{g}_{\GA}(B,N)\ar[r]^{\varepsilon^{g}(B,N)}&
\textrm{Ext}_{R}^{g}(B,N)\ar[r]&\widetilde{\textrm{Ext}}^{g}_{\AAB}(B,N)\ar[r]&0.}$}

\noindent By Lemmas \ref{lem:5.1''} and \ref{lem:5.1'''}, we have the desired commutative diagram in Proposition \ref{thm:C'}.
\end{proof}

\begin{lem}\label{cor:4.4}Let $M$ be a complex with {\rm $\GAT$-dim$M$ $<\infty$}. Then the following are equivalent:
\begin{enumerate}
\item {\rm $\mathcal{A}\textrm{-dim}M$ = $\GAT$-dim$M$};
\item {\rm $\widetilde{\textrm{Ext}}_{\AAB}^i(M,X)=0$} for all $i\in{\mathbb{Z}}$ and all $dg$-$\B$ complexes $X$;
\item {\rm $\widetilde{\textrm{Ext}}_{\AAB}^i(M,N)=0$} for some $i\in{\mathbb{Z}}$ and all $R$-modules $N\in{\B}$.
\end{enumerate}
\end{lem}
\begin{proof}$(1)\Ra(2)$ holds by \cite[Theorem 1.1(1)]{HuDing}.

$(2)\Ra (3)$ is trivial.

$(3)\Ra (1)$. Note that $M$ is a complex of finite Gorenstein $\A$ dimension by hypothesis. Then there exists a
complete $\A$ resolution $T\s{\tau}\ra\mathcal{QR}M\ra \mathcal{R}M\leftarrow M$ of $M$ by Proposition \ref{thm:3.9}. Let $l:{\rm C}_{i}(T)\ra T_{i-1}$ be the canonical injection and $N$ an $R$-module in $\B$. Then $\h_{-i}(\Hom(T,N))=0$ by (3) and Theorem \ref{theorem:3.6}, and so we have the following exact sequence
$$\xymatrix@C=35pt@R=35pt{ \Hom(T_{i-1},N)\ar[rr]^{\textrm{Hom}_{R}\textrm{(}\partial^{T}_{i},N\textrm{)}}
&&\Hom(T_{i},N)\ar[rr]^{\textrm{Hom}_{R}\textrm{(}\partial^{T}_{i+1},N\textrm{)}}&&\Hom(T_{i+1},N).
    }$$
It is easy to check that $\Hom(l,N):\Hom(T_{i-1},N)\ra \Hom({\rm C}_{i}(T),N)$ is epic. Note that $0\ra {\rm C}_{i}(T)\s{l}\ra T_{i-1}\ra {\rm C}_{i-1}(T)\ra 0$ is an exact sequence of $R$-modules such that  ${\rm C}_{i}(T)$ is a Gorenstein $\A$-module. Then we have an exact sequence
$$\Hom(T_{i-1},N)\ra \Hom({\rm C}_{i}(T),N)\ra \Ext^1({\rm C}_{i-1}(T),N)\ra \Ext^1(T_{i-1},N)=0.$$
Thus we have $\Ext^1({\rm C}_{i-1}(T),N)=0$ for any $N\in{\B}$, and hence ${\rm C}_{i-1}(T)\in{\A}$. Consequently, $\textrm{C}_{j}(T)\in{\A}$ for all $j\geqslant i $. Let $j$ be an integer such that $j\geqslant$ max\{$i$, $\GAT$-dim$M$\}. Then $\textrm{C}_{j}(\mathcal{QR}M)\in{\A}$.
Thus $\mathcal{A}\textrm{-dim}M<\infty$ by \cite[Theorem 3.3]{YD}, and so $\GAT\textrm{-dim}M=\mathcal{A}\textrm{-dim}M<\infty$ by Corollary \ref{cor:4.3}. This completes the proof.
\end{proof}

\begin{lem}\label{cor:1.7} Let $M$ be an $R$-module with {\rm $\GAT$-dim$M$ $<\infty$}. Then the following are equivalent:
\begin{enumerate}
\item {\rm $\mathcal{A}\textrm{-dim}M$ = $\GAT$-dim$M$};
\item {\rm $\varepsilon^{n}(B,N):\textrm{Ext}^{n}_{\GA}(B,N)\ra \textrm{Ext}^{n}_{R}(B,N)$} is an isomorphism for all $n\in\mathbb{Z}$ and any $R$-module $N$ in $\B$.
\end{enumerate}
\end{lem}
\begin{proof}
$(1)\Ra(2)$. By hypothesis, there is a nonnegative integer $g$ such that {\rm $\mathcal{A}\textrm{-dim}M$ = $\GAT$-dim$M$ $\leqslant g$} with $g$ an integer.
It follows from Lemma \ref{lem:5.1} that {\rm $\mathcal{A}\textrm{-dim}B$ = $\GAT$-dim$B$ $\leqslant g$}. Let $N$ be an $R$-module in $\B$. One easily checks that $\textrm{Ext}^{0}_{\GA}(M,N)\cong \Hom(M,N)\cong\textrm{Ext}^{0}_{R}(M,N)$ and $\textrm{Ext}^{n}_{\GA}(M,N)=0=\textrm{Ext}^{n}_{R}(M,N)$ for $n<0$ or $n>g$.
For $1\leqslant n\leqslant g$, $\varepsilon^{n}(M,N):\textrm{Ext}^{n}_{\GA}(M,N)\ra \textrm{Ext}^{n}_{R}(M,N)$ is an isomorphism by Proposition \ref{thm:C'} and Lemma \ref{cor:4.4}. So (2) follows.

$(2)\Ra(1)$ holds by  and Proposition \ref{thm:C'} and Lemma \ref{cor:4.4}.
\end{proof}

We end this paper with the proof of Theorem \ref{thm:4.4} as follows.
\begin{para}\label{proofof1.4}{\bf Proof of Theorem \ref{thm:4.4}.}
(1). The result holds by Corollary \ref{cor:4.3}.

(2). $(a)\Leftrightarrow(b)\Leftrightarrow(c)$ hold by Lemma \ref{cor:4.4}.

$(a)\Ra(d)$. Let $\mathcal{QR}M\ra \mathcal{R}M\leftarrow M$ be a cofibrant-fibrant resolution of $M$. Then there is an integer $n$ such that $\textrm{C}_{n}(QRM)\in{\AB}$ by \cite[Theorem 3.3]{YD}. So (d) follows from Lemma \ref{cor:1.7}.

$(d)\Ra(a)$. Let $\mathcal{QR}M\ra \mathcal{R}M\leftarrow M$ be a cofibrant-fibrant resolution of $M$. Then there is an integer $s$ such that $\textrm{C}_{i}(QRM)\in{\B\cap\GA}$ for all $i>s$ by hypothesis and Proposition \ref{thm:3.9}. Thus there is an integer $n$ with $n>s$ such that
$$\varepsilon^{i}(\textrm{C}_{n}(\mathcal{QR}M),N):\textrm{Ext}^{i}_{\GA}(\textrm{C}_{n}(\mathcal{QR}M),N)\ra \textrm{Ext}^{i}_{R}(\textrm{C}_{n}(\mathcal{QR}M),N)$$ is an isomorphism for all $i\in{\mathbb{Z}}$ and any $R$-module $N$ in $\B$. It follows from  Lemma \ref{cor:1.7} that
{\rm$\mathcal{A}\textrm{-dim}\textrm{C}_{n}(\mathcal{QR}M)<\infty$}. Thus $\mathcal{A}\textrm{-dim}M<\infty$ by \cite[Theorem 3.3]{YD}, and so (a) holds by Lemma \ref{cor:4.4}. This completes the proof.
\hfill$\Box$
\end{para}

\bigskip \centerline {\bf ACKNOWLEDGEMENTS}
\bigskip
This research was partially supported by NSFC (11671069, 11771202, 11771212), Qing Lan Project of Jiangsu Province and Jiangsu Government Scholarship for Overseas Studies (JS-2019-328). 

\end{document}